\documentclass[11pt]{amsart}

\usepackage[T1]{fontenc}
\usepackage{lmodern}
\usepackage[expansion=false]{microtype}
\usepackage{amsmath,amssymb,mathtools}
\usepackage{geometry}
\usepackage{enumitem}
\usepackage{xcolor}
\usepackage[colorlinks=true,linkcolor=blue!45!black,citecolor=blue!45!black,urlcolor=blue!45!black,pdftitle={Quantum expanders and dimension-free commutator bounds},pdfauthor={Tuan Tran}]{hyperref}
\allowdisplaybreaks
\newtheorem{theorem}{Theorem}[section]
\newtheorem{lemma}[theorem]{Lemma}
\newtheorem{proposition}[theorem]{Proposition}
\newtheorem{corollary}[theorem]{Corollary}
\newtheorem{remark}[theorem]{Remark}

\newcommand{\Mn}{M_n(\mathbb C)}
\newcommand{\Gn}{GL_n(\mathbb C)}
\newcommand{\Un}{\mathrm U(n)}
\newcommand{\Tr}{\operatorname{Tr}}

\newcommand{\diag}{\operatorname{diag}}
\newcommand{\rank}{\operatorname{rank}}
\newcommand{\ran}{\operatorname{ran}}

\newcommand{\C}{\mathbb C}

\newcommand{\eps}{\varepsilon}
\newcommand{\norm}[1]{\lVert #1\rVert}
\newcommand{\abs}[1]{\lvert #1\rvert}
\newcommand{\comm}[2]{[#1,#2]}
\newcommand{\dOne}{d_1}

\newcommand{\muc}{\mu_0}
\newcommand{\Cgl}{C_{\mathrm{glue}}}
\newcommand{\CCau}{C_{\mathrm{Cau}}}

\newcommand{\cabs}{c_{\mathrm{abs}}}
\newcommand{\Kpoin}{K_{\mathrm P}}
\newcommand{\CRic}{C_{\mathrm R}}
\newcommand{\Kexp}{K_{\mathrm{exp}}}

\title[Quantum expanders and commutator bounds]{Quantum expanders and\\ dimension-free commutator bounds}
\author{Tuan Tran}

\address{School of Mathematical Sciences, University of Science and Technology of China, Hefei, Anhui 230026, China}
\thanks{Supported by the Excellent
Young Talents Program (Overseas) of the National Natural Science Foundation of China under Grant No.
GG0010007003.}
\email{trantuan@ustc.edu.cn}

\begin{document}

\begin{abstract}
We prove that every traceless real or complex matrix $A$ can be written as
$A=BC-CB$, with factors of the same size and over the same field satisfying
$\|B\|\,\|C\|\le K\|A\|$, where $K$ is an absolute constant.  The proof combines an
approximate-rank dichotomy with stable commutator representations obtained
from quantum expansion.
\end{abstract}

\maketitle
\pagestyle{plain}

\section{Introduction}\label{sec:intro-main}

A classical theorem of Shoda \cite{Shoda}, extended to arbitrary fields by
Albert and Muckenhoupt \cite{AlbertMuckenhoupt}, characterizes commutators
$[B,C]=BC-CB$ as exactly the matrices with zero trace.  The quantitative
question asks how large the factors must be.  Throughout, $\|\cdot\|$ denotes
the operator norm.  For a nonzero traceless $A\in\Mn$, define
\[
 \mu(A)=\frac{1}{\|A\|}
 \inf\{\|B\|\,\|C\|:B,C\in\Mn,\ A=[B,C]\},
\]
and put $\mu(0)=0$.  The inequality $\|[B,C]\|\le2\|B\|\,\|C\|$ gives
$\mu(A)\ge1/2$ for $A\ne0$.  Johnson, Ozawa, and Schechtman \cite{JOS}
asked whether $\mu(A)$ is upper-bounded by an absolute constant.

Johnson, Ozawa, and Schechtman connected this question with Anderson's paving
formulation of the Kadison--Singer problem \cite{Anderson}.  A paving
partitions the coordinates so that the associated diagonal compressions have
small operator norm; a multi-paving uses one partition for several matrices.
Fillmore's hollowization theorem \cite{Fillmore} makes a traceless matrix
unitarily equivalent to a matrix with zero diagonal.  Its coordinate
compressions are then traceless, allowing the commutator equation to be
solved recursively on smaller blocks.  This method gave
$\mu(A)\le K_\varepsilon n^\varepsilon$ for every $\varepsilon>0$
\cite[Theorem~1]{JOS}.  Following the solution of Kadison--Singer by Marcus,
Spielman, and Srivastava \cite{MSS}, Ravichandran and Srivastava obtained
optimal-order multi-paving bounds and simplified the commutator argument,
proving $\mu(A)\le300\exp(9\sqrt{\log n})$ \cite[Section~6]{RS}.

We prove that $\mu(A)$ is bounded by an absolute constant and obtain the
same conclusion with real factors for real matrices.

\begin{theorem}\label{thm:main}
There is an absolute constant $K>0$ such that, for
$\mathbb F\in\{\mathbb R,\mathbb C\}$, every traceless matrix
$A\in M_n(\mathbb F)$ admits a representation $A=[B,C]$ with
$B,C\in M_n(\mathbb F)$ and $\|B\|\,\|C\|\le K\|A\|$.
\end{theorem}

There are two main parts to the argument.  The first explains the structure
of a possible counterexample.  A matrix close to one of small rank admits
smaller traceless compressions whose norms decrease enough to close the
induction.  A minimal counterexample must therefore have large approximate
rank (Theorem~\ref{thm:rank-exclusion}).  In even dimension, the resulting
spectral mass gives a substantial trace imbalance between two subspaces of
equal dimension.  The balanced trace-gap theorem
(Theorem~\ref{thm:trace-gap}) converts this imbalance into uniformly
invertible off-diagonal blocks on finitely many compressions.  Thus failure
of the compression argument forces the structure needed for a uniform bound.

The second part constructs commutator representations that remain stable
under perturbation.  We call a two-by-two block matrix \emph{graded} if
its two diagonal blocks are zero.  For such matrices with an invertible
off-diagonal block, we modify the factors without changing their commutator,
using
quantum expanders to control the linearized equation
(Proposition~\ref{prop:graded}).  Small traceless perturbations can then be
absorbed by small changes in the factors, with bounds independent of the
dimension.  The analytic inputs are Hastings's quantum-expander theorem
\cite{Hastings} and Ricard's estimate for the noncommutative Mazur map
\cite{Ricard}; the construction uses them to obtain uniformly regular
representations of this structured class of matrices.

The choice of norm matters.  Angel and Schechtman \cite{AngelSchechtman}
showed that the corresponding estimate for
$\|B\|\,\|C\|_2/\|A\|_2$, with $\|\cdot\|_2$ the Hilbert--Schmidt norm,
has optimal order $\sqrt{\log n}$.  Thus the uniform operator-norm theorem
does not extend to that mixed-norm problem.

Theorem~\ref{thm:main} also has a standard consequence for tracial matrix
ultraproducts.  Johnson, Ozawa, and Schechtman explicitly observed this
implication for Wright factors \cite[Concluding remarks, item~4]{JOS}.
We record the consequence below, keeping track of the inherited norm
bound.  Fix $\mathbb F\in\{\mathbb R,\mathbb C\}$ and a nonprincipal
ultrafilter $\omega$ on $\mathbb N$.  Let
$\mathcal M=\prod_\omega M_{n_k}(\mathbb F)$ be the quotient of the bounded
sequences $(a_k)$ by those satisfying
$\lim_\omega\tau_k(a_k^*a_k)=0$, where $\tau_k=n_k^{-1}\Tr$ and $*$ is the
transpose over $\mathbb R$ and the conjugate transpose over $\mathbb C$.
The induced trace is $\tau_\omega([(a_k)])=\lim_\omega\tau_k(a_k)$.
Over $\mathbb C$, the choice $n_k=k$ gives the Wright factor \cite{Wright}.

\begin{corollary}\label{cor:ultraproduct}
For either field $\mathbb F$, let $T\in\mathcal M$ with
$\tau_\omega(T)=0$.  Then $T=[B,C]$ for some
$B,C\in\mathcal M$ with $\|B\|\,\|C\|\le K\|T\|$, where $K$ is the constant
of Theorem~\ref{thm:main}.
\end{corollary}

For Wright factors, Pearcy and Topping \cite{PearcyTopping} established the
qualitative assertion for self-adjoint elements, and Dykema and Skripka
\cite[Theorem~2.2]{DykemaSkripka} extended it to normal elements.  Wen, Fang,
and Yao \cite[Theorem~4.1 of the preprint]{WenFangYao} obtained the qualitative conclusion
for all trace-zero elements of arbitrary complex $\mathrm{II}_1$ factors.
Their theorem does not assert a uniform bound on the factors.  The bound
in Corollary~\ref{cor:ultraproduct} follows from Theorem~\ref{thm:main} by
the standard passage to the quotient, explained in
Section~\ref{sec:ultraproducts}.

For the proof, it is convenient to use the homogeneous cost
\[
 \muc(A)=\inf\{\|B\|\,\|C\|:B,C\in\Mn,\ A=[B,C]\}
 \quad (A\in\Mn,\ \Tr A=0).
\]
Thus $\muc(A)=\|A\|\mu(A)$ for $A\ne0$ and $\muc(0)=0$.  Both costs are
defined using complex factors.  The real case will be handled separately.

\medskip
\noindent\textbf{Proof overview.}
We organize the overview around a matrix of smallest dimension for which
a uniform commutator bound could fail.  First we show that it cannot be
close to low rank.  We then turn this constraint into a useful block
structure and construct norm-controlled commutators for that structure.
The construction passes through graded matrices, whose representations
can be chosen to remain stable under small perturbations.  We begin with
the estimate that allows us to pass from diagonal compressions back to
the full matrix.

The recursive construction of Johnson, Ozawa, and Schechtman \cite{JOS}
is explained in three steps by Ravichandran and Srivastava
\cite[Section~6]{RS}: partition a hollow matrix into smaller diagonal
compressions, represent those compressions as commutators, and then solve
for the missing off-diagonal blocks.  We use the same assembly principle,
allowing general commutator factors on the diagonal compressions.

Here is how the last step works.  Suppose $A_{jj}=[R_j,S_j]$ for
$j=1,\ldots,r$, and normalize $\|S_j\|\le1$ by reciprocal rescaling.
For a positive number $t$ and complex numbers $z_j$, the identity
\[
 A_{jj}=[t^{-1}R_j,z_jI+tS_j]
\]
preserves each diagonal commutator.  Set $C_j=z_jI+tS_j$ and
$C=\bigoplus_j C_j$.  To obtain $A=[B,C]$, take $B_{jj}=t^{-1}R_j$
and solve $B_{ij}C_j-C_iB_{ij}=A_{ij}$ for $i\ne j$.
Choose the $z_j$ on a square grid in a bounded region, with spacing of
order $r^{-1/2}$, and take $t$ to be a small multiple of that spacing.
Then $\|C\|$ stays bounded, while the diagonal blocks of $B$ grow by a
factor of order $\sqrt r$.

To complete the assembly at this scale, we must bound the full
off-diagonal part of $B$ by an absolute constant times
$\sqrt r\,\|A\|$.  This requires controlling
the solutions of the block equations together in operator norm.
Section~\ref{sec:gluing} obtains that estimate using a lattice Cauchy
multiplier.  This is an instance of the classical Fourier method for
Sylvester equations; see Bhatia, Davis, and McIntosh \cite{BDM} and
Bhatia and Rosenthal \cite[Sections~9--10]{BhatiaRosenthal}.
We give a heat-kernel proof of the lattice estimate.  Combining it with
the diagonal bound above yields
\[
 \muc(A)\le\Cgl\sqrt r
 \left(\|A\|+\max_j\muc(A_{jj})\right)
\]
for any partition into $r$ traceless diagonal compressions.

This estimate tells us how small the diagonal compressions must be for
induction to succeed.  If the smaller blocks
satisfy $\muc(A_{jj})\le K\|A_{jj}\|$ and
$\|A_{jj}\|\le\eta\|A\|$, the bound above becomes
$\muc(A)\le\Cgl\sqrt r(1+K\eta)\|A\|$.
We need $\Cgl\sqrt r\,\eta<1$ to recover the bound $K\|A\|$ for a
sufficiently large $K$.  General paving gives only a bound of order
$r^{-1/2}$ for $\eta$, so increasing $r$ alone does not make the
coefficient of $K$ small.
A reduction of order $1/r$ would make that coefficient small as $r$
grows.  We obtain this stronger reduction for matrices sufficiently close
to low rank.  Such matrices therefore cannot be minimal counterexamples;
we next use this observation to force a structural constraint on any
possible counterexample.

\smallskip
\noindent\emph{Compression and trace imbalance.}
Fix a sufficiently large absolute $K$ and let $A$ be a norm-one traceless
matrix of minimal dimension with $\muc(A)>K$.  Every smaller traceless
matrix $D$ satisfies $\muc(D)\le K\|D\|$.  Thus it suffices to find
smaller traceless compressions of $A$ whose norms meet the requirement
above.

To see how low rank supplies this reduction, suppose $A=F+R$, where
$\rank F$ is small relative to $n$ and $R$ has small operator norm.
Let $P$ project onto $\ran F+\ran F^*$, so that $F=PFP$ and $P$ has small
normalized rank.  The identity $F=PFP$ means that small compressions of
$P$ also give small compressions of $F$.  Damm and Fa\ss bender's
simultaneous hollowization theorem \cite[Section~2.4]{DammFassbender}
makes $A$ hollow and spreads the diagonal of $P$ evenly outside two
exceptional coordinates.  The vectors $Pe_i$ then have small squared
norms on the remaining coordinates.  Marcus--Spielman--Srivastava vector
partitioning \cite[Corollary~1.5]{MSS} gives compressions of $P$, and hence
of $F$, with norms of order $1/r$, once the normalized rank is sufficiently
small.  The remainder stays small under compression, and the two
exceptional coordinates give zero singleton compressions.

For a fixed sufficiently large $r$, this $1/r$ gain compensates for the
$\sqrt r$ assembly cost.  Choosing the rank fraction and the remainder
small enough therefore contradicts $\muc(A)>K$.  Consequently, a minimal
counterexample must satisfy
\[
 \rank_{\varepsilon_0}(A)>\delta_0n,
 \qquad \|A\|_2^2\ge c_0n,
\]
for absolute $\varepsilon_0,\delta_0,c_0>0$.  Here
$\rank_\varepsilon(A)$ is the least rank of a matrix within distance
$\varepsilon\|A\|$ of $A$ in operator norm.

The Hilbert--Schmidt lower bound gives the spectral mass needed for the
other part of the argument.  Assume for now that $n$ is even.  Since
$\|A\|=1$, one of $\operatorname{Re}A$ and $\operatorname{Im}A$ has trace
norm comparable to $n$.  Projecting onto the eigenspaces for its largest
$n/2$ eigenvalues gives an orthogonal projection $P$ with
$|\Tr(PA)|\ge cn$.  Relative to the two halves, the diagonal blocks of
$A$ have normalized traces $\alpha$ and $-\alpha$, with $|\alpha|\ge c$.
Thus the failure of the compression argument forces a definite trace
imbalance between two subspaces of equal dimension.

We next transfer this trace imbalance to an off-diagonal block.  Write
$A=\left(\begin{smallmatrix}A_0&X\\Y&D\end{smallmatrix}\right)$ relative
to the two halves.  Hollowizing $A_0-\alpha I$ and $D+\alpha I$ makes the
diagonals of $A_0$ and $D$ constant.  We then mix paired coordinates by
the unitary matrix
\[
 W_Z=2^{-1/2}\begin{pmatrix}I&Z\\-Z^*&I\end{pmatrix},
 \qquad Z=\diag(z_1,\ldots,z_{n/2}),\qquad |z_i|=1.
\]
If $x_i=X_{ii}$ and $y_i=Y_{ii}$, the $i$th diagonal entry of the
upper-right block of $W_Z^*AW_Z$ is $\alpha z_i+x_i/2-y_i z_i^2/2$.
The phases let us avoid cancellation with the original off-diagonal
entries: the mean squared modulus over $|z_i|=1$ is
$|\alpha|^2+(|x_i|^2+|y_i|^2)/4$, since the three Fourier modes are
orthogonal.  Some choice of $z_i$ therefore makes this entry at least
$|\alpha|$ in modulus.  Paving reduces the remaining hollow part below
$|\alpha|/2$ on each compression, so every resulting link is uniformly
invertible.  Each coordinate pair had trace zero before the mixing and
retains it afterward; hence the compressions formed from these pairs are
traceless.  Since $|\alpha|$ has an absolute lower bound, the number of
compressions is also bounded by an absolute constant.  The assembly
estimate therefore reduces the problem to a uniform commutator bound
for traceless matrices with a uniformly invertible off-diagonal block.

\smallskip
\noindent\emph{Stable commutator representations.}
To prove that bound, we first treat matrices with zero diagonal blocks
and show that their representations persist under small perturbations.
This gives a model to which the matrices just obtained can be reduced.
For $\kappa\ge1$, consider a graded matrix
\[
 T=\begin{pmatrix}0&P\\Q&0\end{pmatrix},
 \qquad P,Q\in M_m(\mathbb C),\qquad
 \|T\|\le1,\qquad \|Q^{-1}\|\le\kappa.
\]
The zero diagonal blocks give the explicit representation $T=[G,C_0]$,
where $G=\diag(I,-I)$ and $C_0=GT/2$.  To make this representation stable,
we need to correct small errors in $T$ by small changes in the factors.
The linearized correction equation is $[X,C]+[B,Y]=E$.  We may add any
block-diagonal matrix to $C_0$ without changing $[G,C_0]$; we use this
freedom to choose factors $B,C$ for which every traceless $E$ admits
corrections satisfying $\|X\|+\|Y\|\le C_\kappa\|E\|$.

Hahn--Banach duality identifies the obstruction to such an estimate:
a matrix that nearly commutes with both $B^*$ and $C^*$ relative to its
distance from the scalars.  These quantities are measured in trace norm,
the dual of the operator norm used to measure the error.  Commutation
with the grading controls the off-diagonal blocks of this witness.
The invertible link then makes its two diagonal blocks close after
conjugating one of them by $Q^*$.  We place expander operators in the
available block-diagonal part of $C$ to force the remaining common part
close to a scalar.  The estimate that makes this possible is
\[
 d_1(Z)\le\Kpoin\bigl(\|[Z,U]\|_1+\|[Z,V]\|_1\bigr),
 \qquad d_1(Z)=\inf_{\lambda\in\mathbb C}\|Z-\lambda I\|_1,
\]
for every $Z\in M_m(\mathbb C)$, with suitable unitaries $U,V$ in that
dimension.  Here $\|\cdot\|_1$ denotes trace norm.

To obtain this trace-class estimate, we start with the Hilbert--Schmidt
spectral gap supplied by Hastings's quantum-expander theorem
\cite{Hastings}.  For a Hermitian witness $Z$, subtract an eigenvalue
median to obtain $H$.  Its signed square root has squared
Hilbert--Schmidt norm $\|H\|_1=d_1(Z)$.  The median ensures that a fixed
proportion remains after removing the scalar part, so the spectral gap
controls this quantity.  Ricard's Mazur-map estimate \cite{Ricard}
bounds the commutators of the signed square root in terms of those of
$H$, giving the desired trace-norm bound.  Applying this argument to the
real and imaginary parts yields the inequality for general $Z$.

With the dual obstruction controlled, Hahn--Banach gives the linear
correction estimate.  The full commutator equation has one additional
term, $[X,Y]$, which is quadratic in the corrections.  An iteration
absorbs this term when the traceless error $E$ is sufficiently small.
Thus the chosen representation of $T$ extends to one of $T+E$, with
changes in the factors bounded by $C_\kappa\|E\|$.
Proposition~\ref{prop:graded} isolates this regularity statement, with
constants independent of $m$.

We can now return to the matrices with invertible links.  The remaining
task is to make their diagonal blocks small while preserving the link,
so that the stability estimate applies.  For a traceless matrix
$A=\left(\begin{smallmatrix}A_0&X\\Y&D\end{smallmatrix}\right)$ with
$s_{\min}(X)\ge\tau\|A\|$, where $s_{\min}$ is the smallest singular
value, a triangular similarity with condition number depending only on
$\tau$ makes $A_0$ zero while leaving $X$ unchanged.  The new lower-right
block is traceless.  A polar decomposition makes the link positive,
with a lower bound preserved by compression.
A common change of basis in the two halves makes the remaining diagonal
block hollow and preserves positivity of the link.  Ravichandran--Srivastava
paving makes that block small on each compression while the link stays
invertible.  After exchanging the two halves, the resulting matrices are
close to graded matrices of the form above.  The paving accuracy depends
only on $\tau$, so the stability estimate and reassembly give
$\muc(A)\le C_\tau\|A\|$.  Applying this bound to the finitely many
compressions obtained from the trace gap, and assembling once more,
gives an absolute bound for the original counterexample.  This
contradicts $\muc(A)>K$ when $K$ is sufficiently large.

This rules out even-dimensional counterexamples.  In odd dimension,
hollowization and a scalar resolvent construction give
\[
 \sqrt{\muc(A)}\le\sqrt{\muc(D)}+1
\]
for an even-dimensional traceless compression $D$ of codimension one.
If $\muc(A)>K$, then $\muc(D)>(\sqrt K-1)^2\ge3K/4$ for sufficiently
large $K$.  This is why the approximate-rank exclusion theorem is stated
with the threshold $3K/4$: it still applies to $D$, whose norm is at most
one.  The even-dimensional estimate bounds its cost by an absolute
constant, and the square-root inequality then does the same for $A$.
All partitions are used at accuracies independent of the dimension.

This completes the complex argument.  To pass to real matrices, we use
an orthogonal complex structure to separate complex-linear and
antilinear parts.  A shifted commutator equation then combines them into
real factors of the original size.

\medskip
\noindent\textbf{Organization.}
The paper develops these tools before applying the minimal-dimension
argument.  Section~\ref{sec:gluing}
provides the estimates used whenever block representations are reassembled.
Section~\ref{sec:regularity} constructs stable representations of graded
matrices; Section~\ref{sec:links} extends this construction first to
matrices with an invertible link and then to matrices with a trace gap.
At that point the trace gap is a sufficient condition for a uniform bound.
Section~\ref{sec:rank} shows how the smaller-dimensional bounds force
enough spectral mass to produce this condition in even dimension.
Section~\ref{sec:main} combines these results, handles odd dimensions,
deduces the real case, and proves the ultraproduct corollary.
\section{Preliminaries and gluing tools}\label{sec:gluing}

Both branches of the proof end by assembling commutator representations
of diagonal compressions.  In the low-rank branch, the reduction in their
norms must compensate for the assembly cost; in the trace-gap branch, their
representations already have dimension-free bounds.  To recover the full
matrix, we must solve for the off-diagonal blocks with controlled factor
norms.  We use the scalar-shift
and Sylvester-equation method of Johnson, Ozawa, and Schechtman
\cite{JOS} and Ravichandran and Srivastava \cite[Section~6]{RS}.
The underlying solvability principle goes back to Rosenblum
\cite{Rosenblum}.

We state the two assembly estimates first.  After recording the elementary
reductions and the hollowization results of Fillmore and Damm--Fa\ss bender,
we prove the estimates using a lattice form of the classical Fourier method
for Sylvester equations \cite{BDM,BhatiaRosenthal}.

\noindent\textbf{Notation.}
The matrix arguments in Sections~\ref{sec:gluing}--\ref{sec:rank} use
finite-dimensional complex Hilbert spaces and orthogonal direct sums.
Hilbert spaces are denoted by
calligraphic letters, such as $\mathcal H$, $\mathcal K$, and $\mathcal V$.
We write $\mathcal B(\mathcal K,\mathcal H)$ for the space of linear maps
from $\mathcal K$ to $\mathcal H$, put
$\mathcal B(\mathcal H)=\mathcal B(\mathcal H,\mathcal H)$, and write
$I_{\mathcal H}$ for the identity on $\mathcal H$, abbreviated to $I$ when
the space is clear.

For an orthogonal decomposition
$\mathcal H=\bigoplus_{j=1}^s\mathcal H_j$, with $P_j$ the orthogonal
projection onto $\mathcal H_j$, the $(i,j)$ block of
$X\in\mathcal B(\mathcal H)$ is
\[
 X_{ij}=P_iX|_{\mathcal H_j}\in\mathcal B(\mathcal H_j,\mathcal H_i).
\]
For orthogonal projections $P,Q\in\mathcal B(\mathcal H)$, a corner
$PX|_{Q\mathcal H}:Q\mathcal H\to P\mathcal H$ is identified with its
extension by zero $PXQ$ when used in an ambient operator product.  For a
compression $PXP$, its inverse, least singular value, and commutator cost are
always understood on $P\mathcal H$, not on the ambient space.

We use the operator norm unless a Schatten subscript is displayed.  For
$T\in\mathcal B(\mathcal K,\mathcal H)$, put $|T|=(T^*T)^{1/2}$ and
\[
 \|T\|_1=\Tr_{\mathcal K}|T|,
 \qquad
 \|T\|_2=\bigl(\Tr_{\mathcal K}(T^*T)\bigr)^{1/2}.
\]
These conventions apply also to rectangular blocks.  The trace is
unnormalized, while $\tau(X)=n^{-1}\Tr X$ for $X\in\Mn$.
For $T\in\mathcal B(\mathcal K,\mathcal H)$ with
$\mathcal K\ne\{0\}$, set
\[
 s_{\min}(T):=\inf_{\substack{v\in\mathcal K\\\|v\|=1}}\|Tv\|.
\]
For Hermitian matrices $X,Y$ we write $X\preceq Y$ if $Y-X$ is positive
semidefinite.  For a subset $S$ of the index set of a fixed orthonormal basis
$(e_i)$, we write $P_S$ for the orthogonal projection onto
$\operatorname{span}\{e_i:i\in S\}$, and call such a projection a
coordinate projection.  Constants denoted by $C,c$ are positive and absolute
unless their dependence is indicated.
For a positive integer $m$, write $[m]=\{1,\ldots,m\}$.
Partitions into a prescribed number of classes may include empty classes,
which are discarded when applying results to the resulting compressions.

\medskip
The first gluing estimate recombines a fixed number of traceless diagonal
compressions while controlling all off-diagonal blocks simultaneously.

\begin{lemma}[Finite-block gluing]\label{lem:finite-gluing}
There is an absolute constant $\Cgl>0$ with the following property.  Let $\C^n=\mathcal H_1\oplus\cdots\oplus \mathcal H_s$ be an orthogonal decomposition, let $A\in \Mn$, and let $A_j\in\mathcal B(\mathcal H_j)$ be the compression of $A$ to $\mathcal H_j$.  If every $A_j$ is traceless, then
\[
 \muc(A)\le \Cgl\sqrt{s}
 \Bigl(\norm{A}+\max_{1\le j\le s}\muc(A_j)\Bigr).
\]
\end{lemma}

The second estimate treats a zero diagonal block by the same scalar-shift
method.  Its square-root form will let us remove one coordinate in the
odd-dimensional case without a multiplicative loss.

\begin{lemma}[One-sided square-root gluing]\label{lem:one-sided}
Let $\C^n=\mathcal H_1\oplus \mathcal H_2$ and let
$A\in\Mn$ have block form
$A=\left(\begin{smallmatrix}0&X\\Y&D\end{smallmatrix}\right)$ relative to
this decomposition, with $D\in\mathcal B(\mathcal H_2)$ traceless.  Then
\[
 \sqrt{\muc(A)}
 \le\sqrt{\muc(D)}
      +\sqrt{\max\{\norm{X},\norm{Y}\}}.
\]
\end{lemma}

\subsection{Basic reductions and hollowization}

We begin with the elementary invariance and rescaling properties of the
homogeneous commutator cost.  These will allow us to normalize matrices and
balance the two commutator factors without further comment.

\begin{lemma}\label{lem:basic-mu}
Let $A\in \Mn$ be a traceless matrix.  Then the following hold.
\begin{enumerate}[label=\textup{(\roman*)},leftmargin=2.2em]
\item For every $t\in\mathbb C$, $\muc(tA)=\abs{t}\muc(A)$.
\item For every $U\in\Un$, $\muc(U^*AU)=\muc(A)$.
\item For every invertible matrix $S \in 
\Gn$, $\muc(S^{-1}AS)\le(\norm{S}\norm{S^{-1}})^2\muc(A)$.
\item For every $\eta>0$ there are $B,C\in\Mn$ with $A=\comm{B}{C}$ and
$\norm{B}=\norm{C}\le\sqrt{\muc(A)+\eta}$.
\end{enumerate}
\end{lemma}

\begin{proof}
Parts \textup{(i)}--\textup{(iii)} follow by rescaling or conjugating a
commutator representation and then taking infima.  For \textup{(iv)}, choose
$A=[B,C]$ with $\|B\|\,\|C\|\le\muc(A)+\eta$.  If $A\ne0$, reciprocal
rescaling of $B$ and $C$ makes their norms equal without changing either the
commutator or the product of the norms.  If $A=0$, take $B=C=0$.
\end{proof}

A matrix is \emph{hollow} if its diagonal is zero.  In a hollow basis,
every coordinate compression remains traceless, which is essential when
we apply an inductive commutator bound.  Part~\textup{(i)} below is
Fillmore's theorem \cite{Fillmore}.  Part~\textup{(ii)} is the simultaneous
hollowization theorem of Damm and Fa\ss bender
\cite[Section~2.4, Proposition~2.13(b)]{DammFassbender}
(Proposition~13(b) in the preprint).  Its two exceptional coordinates will
be isolated in the approximate-rank argument.

\begin{theorem}\label{thm:hollow}
The following assertions hold.
\begin{enumerate}[label=\textup{(\roman*)},leftmargin=2.2em]
\item If $A\in \Mn$ is traceless, there is a unitary matrix $U\in \Un$ such that $U^*AU$ has zero diagonal.
\item If $n\ge2$ and $H_1,H_2,H_3\in \Mn$ are traceless Hermitian matrices, there is a unitary matrix $U\in \Un$ such that $U^*H_1U$ and $U^*H_2U$ have zero diagonal, and the diagonal of $U^*H_3U$ is $(0,\ldots,0,t,-t)$ for some real $t$.
\end{enumerate}
\end{theorem}

For completeness, we recall Damm and Fa\ss bender's proof of
part~\textup{(ii)}.  It uses the convexity theorem of Au-Yeung and Poon
\cite{AuYeungPoon} for the joint numerical range of three Hermitian forms
in dimension at least three, followed by Fillmore's theorem on the
remaining two-dimensional subspace.

\begin{proof}[Proof of Theorem~\textup{\ref{thm:hollow}(ii)}]
We first choose orthonormal vectors $v_1,\ldots,v_{n-2}$ such that
\[
 \langle H_kv_i,v_i\rangle=0
 \quad\text{for }k=1,2,3\text{ and }i=1,\ldots,n-2;
\]
for $n=2$ this step is vacuous.  Suppose $v_1,\ldots,v_{j-1}$ have been chosen with $j\le n-2$, and let $\mathcal V$ be the orthogonal complement of their span, so that $\dim \mathcal V=n-j+1\ge3$.  The compressions of $H_1,H_2,H_3$ to $\mathcal V$ are traceless Hermitian matrices, since the trace of each compression equals the full trace minus $\sum_{i<j}\langle H_kv_i,v_i\rangle=0$.  The joint numerical range
\[
 W=\bigl\{\bigl(\langle H_1v,v\rangle,\langle H_2v,v\rangle,\langle H_3v,v\rangle\bigr):
 v\in \mathcal V,\ \norm{v}=1\bigr\}\subset\mathbb R^3
\]
of three Hermitian forms on a space of dimension at least three is convex \cite{AuYeungPoon}.  Averaging over an orthonormal basis of $\mathcal V$ exhibits $(0,0,0)$ as the barycenter of finitely many points of $W$, so $(0,0,0)\in W$; choose $v_j\in \mathcal V$ accordingly.

Let $\mathcal V_0$ be the orthogonal complement of $v_1,\ldots,v_{n-2}$, a two-dimensional space; as above, the compressions of $H_1,H_2,H_3$ to $\mathcal V_0$ are traceless.  The compression of $H_1+iH_2$ to $\mathcal V_0$ is a $2\times2$ traceless matrix, so by part \textup{(i)} there is an orthonormal basis $w_1,w_2$ of $\mathcal V_0$ in which it is hollow; equivalently, the compressions of $H_1$ and $H_2$ are hollow in this basis, because $\langle(H_1+iH_2)w,w\rangle$ has real part $\langle H_1w,w\rangle$ and imaginary part $\langle H_2w,w\rangle$.  The compression of $H_3$ to $\mathcal V_0$ is a traceless Hermitian $2\times2$ matrix, so its two diagonal entries in the basis $w_1,w_2$ are opposites.  The unitary matrix sending the standard basis to $v_1,\ldots,v_{n-2},w_1,w_2$ has the required properties.
\end{proof}

We conclude the preliminary reductions with four elementary facts about block
compressions, trace norms, and resolvents.  They will be invoked repeatedly in
the two gluing proofs.

\begin{lemma}\label{lem:block-facts}
Let $\mathcal H$ be a Hilbert space.
\begin{enumerate}[label=\textup{(\roman*)},leftmargin=2.2em]
\item For a family $\mathcal P=(P_1,\ldots,P_s)$ of pairwise orthogonal
projections on $\mathcal H$ with $\sum_jP_j=I_{\mathcal H}$, the
block-diagonal map $\mathbb E_{\mathcal P}:\mathcal B(\mathcal H)\to
\mathcal B(\mathcal H)$, defined by
$\mathbb E_{\mathcal P}(X)=\sum_jP_jXP_j$, is contractive in operator
and trace norms.
\item For orthogonal projections $P,Q$ on $\mathcal H$, the corner map
$\mathcal B(\mathcal H)\to\mathcal B(Q\mathcal H,P\mathcal H)$,
$X\mapsto PX|_{Q\mathcal H}$, is contractive in both norms.
\item If $\mathcal H=\mathcal H_1\oplus\mathcal H_2$,
$Z\in\mathcal B(\mathcal H_2,\mathcal H_1)$ and
$W\in\mathcal B(\mathcal H_1,\mathcal H_2)$, then
$M=\left(\begin{smallmatrix}0&Z\\W&0\end{smallmatrix}\right)
\in\mathcal B(\mathcal H)$ satisfies
\[
 \|M\|=\max\{\|Z\|,\|W\|\},
 \qquad
 \|M\|_1=\|Z\|_1+\|W\|_1.
\]
\item For $T\in\mathcal B(\mathcal H)$ with $\|T\|\le1$ and $\eta>0$,
both $(1+\eta)I_{\mathcal H}-T$ and $T-(1+\eta)I_{\mathcal H}$ are
invertible, with inverse norms at most $\eta^{-1}$.
\end{enumerate}
\end{lemma}

\begin{proof}
For (i), let $\zeta_1,\ldots,\zeta_s$ be independent Haar phases and put $D_\zeta=\sum_j\zeta_jP_j$.  Then
\[
 \mathbb E_{\mathcal P}(X)
 =\int_{\mathbb T^s}D_\zeta XD_\zeta^*\,dm(\zeta).
\]
Every conjugation in the integral is an isometry in both norms, so the assertion follows from the triangle inequality.  Part (ii) follows directly from
$\|PXQ\|\le\|P\|\,\|X\|\,\|Q\|$ and the corresponding ideal property of trace norm.

For (iii), if $M=\left(\begin{smallmatrix}0&Z\\W&0\end{smallmatrix}\right)$, then
\[
 M^*M=\begin{pmatrix}W^*W&0\\0&Z^*Z\end{pmatrix}.
\]
Thus the singular values of $M$ are the singular values of $W$ together with those of $Z$, counted with multiplicity.  Taking the largest singular value and the sum of all singular values gives the two formulas.

For (iv), write
\[
 (1+\eta)I_{\mathcal H}-T=(1+\eta)\left(I_{\mathcal H}-\frac{T}{1+\eta}\right).
\]
The Neumann series converges because $\|T/(1+\eta)\|<1$ and gives inverse norm at most
$(1+\eta)^{-1}(1-(1+\eta)^{-1})^{-1}=\eta^{-1}$.  The second inverse differs only by a sign.
\end{proof}

\subsection{The Cauchy multiplier and the gluing estimates}\label{sec:cauchy}

Estimating each off-diagonal block separately would introduce a loss when
the blocks are reassembled.  Instead, we estimate the full block operator
by representing the inverse Cauchy kernel as an average of unitary
conjugations.  This is the Fourier method used by Bhatia, Davis, and
McIntosh \cite{BDM}; see also Bhatia and Rosenthal
\cite[Sections~9--10]{BhatiaRosenthal}.  The heat kernel gives a convenient
self-contained proof for the lattice kernel used here.  Once this
multiplier bound is available, scalar shifts solve all the block equations
at once.

Throughout, $\mathbb T^2=(\mathbb R/2\pi\mathbb Z)^2$ carries normalized Haar
measure $m$, and the Fourier coefficients of $f\in L_1(\mathbb T^2)$ are
$\widehat f(k)=\int_{\mathbb T^2}f(x)e^{-ik\cdot x}\,dm(x)$ for
$k\in\mathbb Z^2$.  For $t>0$ the periodic heat kernel is
\[
 p_t(x)=\sum_{k\in\mathbb Z^2}e^{-t|k|^2}e^{ik\cdot x}
 \quad\text{for }x\in\mathbb T^2,
\]
so that $\widehat p_t(k)=e^{-t|k|^2}$.  Writing
$g_t(x)=(4\pi t)^{-1}e^{-|x|^2/(4t)}$ for the Euclidean heat kernel on
$\mathbb R^2$, Poisson summation gives
$p_t(x)=(2\pi)^2\sum_{k\in\mathbb Z^2}g_t(x+2\pi k)$.

The following standard heat-kernel estimate makes the Fourier
representation integrable in time.  We include the short proof to keep
track of the normalization of Haar measure.

\begin{lemma}\label{lem:heat-gradient}
Let $p_t$ be the periodic heat kernel on $\mathbb T^2$, with $t>0$.
There are absolute constants $C,c>0$ such that
\[
 \|\nabla p_t\|_{L_1(\mathbb T^2)}\le
 \begin{cases}
 Ct^{-1/2}&\text{for }0<t\le1,\\[1mm]
 Ce^{-ct}&\text{for }t\ge1.
 \end{cases}
\]
\end{lemma}

\begin{proof}
For $0<t\le1$, use the periodization $p_t(x)=(2\pi)^2\sum_kg_t(x+2\pi k)$.  Let $F=[-\pi,\pi)^2$ be a fundamental domain, on which $m$ is $(2\pi)^{-2}$ times Lebesgue measure.  The two factors of $(2\pi)^2$ cancel, and Tonelli's theorem gives
\[
 \|\nabla p_t\|_{L_1(\mathbb T^2)}
 =\int_F\left|\sum_{k\in\mathbb Z^2}\nabla g_t(x+2\pi k)\right|dx
 \le\sum_{k\in\mathbb Z^2}\int_F|\nabla g_t(x+2\pi k)|dx
 =\int_{\mathbb R^2}|\nabla g_t(x)|dx.
\]
The last integral equals $Ct^{-1/2}$ by the change of variables $x=\sqrt t\,y$.

For $t\ge1$, the Fourier series is absolutely convergent and
\[
 \|\nabla p_t\|_1\le\|\nabla p_t\|_\infty
 \le\sum_{k\in\mathbb Z^2\setminus\{0\}}|k|e^{-t|k|^2}.
\]
Since $|k|^2\ge1$, one may factor out $e^{-t/2}$ and sum
$|k|e^{-|k|^2/2}$, obtaining the claimed exponential bound.
\end{proof}

Integrating the heat-kernel estimate now yields the dimension-free multiplier
bound that drives finite-block gluing.

\begin{lemma}\label{lem:cauchy}
Let $\mathcal H_1,\ldots,\mathcal H_s$ be Hilbert spaces and let
$(a_1,b_1),\ldots,(a_s,b_s)$ be distinct points of $\mathbb Z^2$.
The linear map $\mathcal C$ from the zero-block-diagonal subspace of
$\mathcal B(\bigoplus_{j=1}^s\mathcal H_j)$ to itself, defined by
\[
 (\mathcal CX)_{ij}=\frac{X_{ij}}{(a_i-a_j)+\mathrm i(b_i-b_j)}
 \quad(i\ne j),
 \qquad
 (\mathcal CX)_{ii}=0,
\]
satisfies $\|\mathcal CX\|\le\CCau\|X\|$ for every $X$,
where $\CCau>0$ is absolute.
\end{lemma}

\begin{proof}
By Lemma~\ref{lem:heat-gradient}, the Bochner integral
\[
 f=\int_0^\infty(-i\partial_x-\partial_y)p_t\,dt
\]
converges in $L_1(\mathbb T^2)$.  Indeed, its $L_1$ norm is bounded by
\[
 \int_0^1 Ct^{-1/2}\,dt+\int_1^\infty Ce^{-ct}\,dt<\infty.
\]
For $(a,b)\neq(0,0)$, differentiation of the Fourier series and Fubini's theorem give
\[
 \widehat f(a,b)
 =\int_0^\infty(a-ib)e^{-t(a^2+b^2)}\,dt
 =\frac{a-ib}{a^2+b^2}
 =\frac1{a+ib},
\]
whereas $\widehat f(0,0)=0$.

For $\theta=(\theta_1,\theta_2)$ define the block-diagonal unitary matrix
\[
 D_\theta\big|_{\mathcal H_j}
 =e^{-i(a_j\theta_1+b_j\theta_2)}I_{\mathcal H_j}.
\]
The $(i,j)$ block of $D_\theta XD_\theta^*$ is
$e^{-i((a_i-a_j)\theta_1+(b_i-b_j)\theta_2)}X_{ij}$, so integrating against $f$ reproduces the Fourier coefficient $\widehat f(a_i-a_j,b_i-b_j)$ in each block:
\[
 \int_{\mathbb T^2}f(\theta)D_\theta XD_\theta^*\,dm(\theta)
 =\mathcal CX;
\]
the diagonal blocks vanish because $X_{ii}=0$ and $\widehat f(0,0)=0$.  Since conjugation by $D_\theta$ is an operator-norm isometry,
\[
 \|\mathcal CX\|
 \le\int_{\mathbb T^2}|f(\theta)|\,\|X\|\,dm(\theta)
 \le\|f\|_1\|X\|.
\]
The proof uses no scalar-block assumption, so it applies verbatim to operator-valued and rectangular blocks.
\end{proof}

With the multiplier estimate in hand, we can solve all off-diagonal Sylvester
equations simultaneously and prove the first gluing lemma.

\begin{proof}[Proof of Lemma~\ref{lem:finite-gluing}]
Let $P_j$ be the orthogonal projection of $\C^n$ onto $\mathcal H_j$.  Relative to the orthogonal decomposition $\C^n=\mathcal H_1\oplus\cdots\oplus \mathcal H_s$, write $A_{ij}=P_iA|_{\mathcal H_j}:\mathcal H_j\longrightarrow \mathcal H_i$. Thus $A_{jj}=A_j$, and the block matrix $(A_{ij})$ is simply the original operator $A$ expressed relative to the chosen projections.  The aim is to retain prescribed commutator representations of the diagonal compressions $A_j$ while solving all off-diagonal equations for $A_{ij}$ at once.

Put $q=\lceil\sqrt s\rceil$ and choose distinct pairs
$(a_j,b_j)\in\{0,\ldots,q-1\}^2$.  Set $z_j=q^{-1}(a_j+ib_j)$. Then $|z_j|\le\sqrt2$.  Fix a small absolute number $\gamma>0$, whose value will be chosen after the Cauchy-multiplier norm is known.

Fix $\xi>0$.  For each $j$, choose
$A_j=[R_j,S_j]$ with $\|R_j\|\,\|S_j\|\le\muc(A_j)+\xi$. If $A_j\neq0$, reciprocal rescaling allows us to impose
$\|S_j\|=\gamma/q$; then
\[
 \|R_j\|\le\frac q{\gamma}
 \left(\max_{\ell}\muc(A_{\ell\ell})+\xi\right).
\]
If $A_j=0$, take $R_j=S_j=0$.  Define
\[
 C_0=\bigoplus_{j=1}^s(z_jI_{\mathcal H_j}+S_j),
 \qquad
 B_0=\bigoplus_{j=1}^sR_j.
\]
Since scalar matrices commute with everything, the $j$th diagonal block of $[B_0,C_0]$ is $[R_j,S_j]=A_j$.

Let $\mathcal X$ be the subspace of block matrices with zero block diagonal.  For $X\in\mathcal X$ set
\[
 (\mathcal TX)_{ij}=(z_j-z_i)X_{ij},
 \qquad
 (\mathcal SX)_{ij}=X_{ij}S_j-S_iX_{ij}
 \quad\text{for }i\neq j.
\]
The inverse of $\mathcal T$ is division by $z_j-z_i$.  Because
\[
 z_j-z_i=q^{-1}\bigl((a_j-a_i)+i(b_j-b_i)\bigr),
\]
Lemma~\ref{lem:cauchy} gives
\[
 \|\mathcal T^{-1}\|_{\mathcal X\to\mathcal X}
 \le \CCau{}q.
\]
Since
\[
 \mathcal SX=X\Bigl(\bigoplus_jS_j\Bigr)
             -\Bigl(\bigoplus_jS_j\Bigr)X,
\]
we have
\[
 \|\mathcal S\|
 \le2\max_j\|S_j\|
 \le\frac{2\gamma}{q}.
\]
Choose $\gamma$ so that $2\CCau\gamma\le1/2$.  Then
$\|\mathcal T^{-1}\mathcal S\|\le1/2$, and the Neumann series gives
\[
 (\mathcal T+\mathcal S)^{-1}
 =(I+\mathcal T^{-1}\mathcal S)^{-1}\mathcal T^{-1},
 \qquad
 \|(\mathcal T+\mathcal S)^{-1}\|
 \le2\CCau{}q.
\]

Put
\[
 A_{\rm off}=A-\bigoplus_{j=1}^sA_j.
\]
By Lemma~\ref{lem:block-facts}(i), the block-diagonal compression has norm at most $\|A\|$, and hence $\|A_{\rm off}\|\le2\|A\|$. There is therefore a unique $X\in\mathcal X$ satisfying
$(\mathcal T+\mathcal S)X=A_{\rm off}$, and $\|X\|\le4\CCau{}q\|A\|$. For $i\neq j$, the block equation reads
\[
 X_{ij}(z_jI+S_j)-(z_iI+S_i)X_{ij}=A_{ij}.
\]
Consequently, with $B=B_0+X$ and $C=C_0$, the diagonal blocks of $[B,C]$ are $A_j$ and the off-diagonal blocks are $A_{ij}$; hence $[B,C]=A$.

Finally,
\[
 \|C\|\le\max_j|z_j|+\max_j\|S_j\|
 \le\sqrt2+\gamma,
\]
and
\[
 \|B\|
 \le\|B_0\|+\|X\|
 \le\frac q{\gamma}
      \left(\max_{\ell}\muc(A_{\ell\ell})+\xi\right)
      +4\CCau{}q\|A\|.
\]
Since $q\le2\sqrt s$, multiplication of the last two estimates gives the asserted bound with an absolute constant.  Letting $\xi\downarrow0$ completes the proof.
\end{proof}

The one-sided estimate uses a simpler resolvent construction and completes the
proof of the two gluing statements.

\begin{proof}[Proof of Lemma~\ref{lem:one-sided}]
Fix $\xi>0$.  Choose $D=[P,Q]$ with
$\|P\|\,\|Q\|\le\muc(D)+\xi$.  If $D\neq0$, reciprocal rescaling permits us to assume $\|Q\|=1$, $\|P\|\le\muc(D)+\xi$; if $D=0$, take $P=Q=0$, and the same inequalities hold.  For a parameter $\eta>0$, set
\[
 C=\begin{pmatrix}(1+\eta)I&0\\0&Q\end{pmatrix},
 \qquad
 B=\begin{pmatrix}0&Z\\W&P\end{pmatrix}.
\]
A direct block multiplication gives
\[
 [B,C]=
 \begin{pmatrix}
 0&ZQ-(1+\eta)Z\\
 (1+\eta)W-QW&[P,Q]
 \end{pmatrix}.
\]
Thus it remains to solve $Z(Q-(1+\eta)I)=X$, $((1+\eta)I-Q)W=Y$. Lemma~\ref{lem:block-facts}(iv) gives solutions with
$\|Z\|\le\|X\|/\eta$ and $\|W\|\le\|Y\|/\eta$.  The diagonal and off-diagonal parts of $B$ satisfy
\[
 \left\|\begin{pmatrix}0&Z\\W&0\end{pmatrix}\right\|
 =\max\{\|Z\|,\|W\|\}
 \le\frac{\max\{\|X\|,\|Y\|\}}{\eta}
\]
by Lemma~\ref{lem:block-facts}(iii), and hence
\[
 \|B\|
 \le\muc(D)+\xi
      +\frac{\max\{\|X\|,\|Y\|\}}{\eta},
 \qquad
 \|C\|\le1+\eta.
\]
We have proved
\[
 \muc(A)
 \le(1+\eta)\left(
       \muc(D)+\xi
       +\frac{\max\{\|X\|,\|Y\|\}}{\eta}
      \right).
\]
If both $\muc(D)+\xi$ and $\max\{\|X\|,\|Y\|\}$ are positive, choose
$\eta=
 \sqrt{\frac{\max\{\|X\|,\|Y\|\}}{\muc(D)+\xi}}$.
The right-hand side then equals
\[
 \left(
   \sqrt{\muc(D)+\xi}
   +\sqrt{\max\{\|X\|,\|Y\|\}}
 \right)^2.
\]
If either quantity vanishes, the same bound follows by taking a limit in $\eta$.  Finally let $\xi\downarrow0$ and take square roots.
\end{proof}
\section{Stable commutator representations}\label{sec:regularity}

The gluing estimates control the cost of reassembling blocks once their
commutator representations are known.  We now construct the representations
that will be used in the trace-gap argument.  The compressions in
Section~\ref{sec:links} will be close to graded matrices with an invertible
link, so we need representations whose factors can be adjusted when the
matrix is perturbed.  Proposition~\ref{prop:graded}(i) gives a uniform
bound for the linearized equation; part (ii) converts it into the
neighborhood estimate needed there.  The freedom to add a block-diagonal
matrix to the second factor lets us improve the linearized equation while
keeping the represented matrix fixed.

We first derive a trace-class Poincar\'e inequality from Hastings's quantum
expanders \cite{Hastings} and Ricard's Mazur-map estimate \cite{Ricard}.
It says that a matrix nearly commuting with two suitable unitaries is close
to a scalar.  We insert these operators into the representation, with a
similarity accounting for the invertible link.  Duality between operator
norm and trace norm then gives norm-controlled solutions of the linearized
equation, and successive corrections absorb the quadratic error.  We state
the analytic estimate and the local regularity proposition before proving
them.

To make the correction problem explicit, for $B,C,X,Y\in\Mn$ we have
\[
 [B+X,C+Y]-[B,C]=[X,C]+[B,Y]+[X,Y].
\]
The first two terms form the linearized equation.  Solving that equation
with corrections controlled by the size of the error leaves only the
quadratic term $[X,Y]$, which can be absorbed by iteration.  The quantity
defined below measures the bound needed for these linear corrections.

Put $M_n^0=\{E\in\Mn:\Tr E=0\}$.  For $Z\in\Mn$, define
\[
 d_1(Z)=\inf_{\lambda\in\mathbb C}\|Z-\lambda I\|_1.
\]
For $B,C\in\Mn$, define
$\mathcal L_{B,C}:\Mn\times\Mn\to M_n^0$ by
$\mathcal L_{B,C}(X,Y)=[X,C]+[B,Y]$, and put
\[
 \Gamma(B,C)=\sup_{Z\in\Mn\setminus\mathbb CI}
 \frac{d_1(Z)}{\|[Z,C^*]\|_1+\|[B^*,Z]\|_1}.
\]
A zero denominator gives $\Gamma(B,C)=\infty$; for $n=1$, set
$\Gamma(B,C)=0$.  A finite value bounds the norm needed to solve the
linearized equation, as Lemma~\ref{lem:linear-surj} will show.  Thus
$\Gamma(B,C)$ measures quantitatively how nearly a nonscalar matrix can
commute with both adjoint factors.  The following consequence of Hastings's
and Ricard's theorems supplies the
unitaries used to obtain such a bound.

\begin{theorem}[Dimension-free $S_1$ Poincar\'e inequality]\label{thm:s1-poincare}
For every $n$, there are unitary matrices $U,V\in\Un$ such that
\[
 \dOne(Z)\le \Kpoin
 \left(\norm{\comm{Z}{U}}_1+\norm{\comm{Z}{V}}_1\right)
\]
for every $Z\in \Mn$, where $\Kpoin$ is absolute.
\end{theorem}

The next proposition uses the freedom in choosing commutator factors to
obtain a uniform bound for $\Gamma$.  This is the regular representation
needed to correct perturbations.

\begin{proposition}[Regular representations of graded matrices]\label{prop:graded}
For every $\kappa\ge1$ there are constants $M_{\kappa},\eta_{\kappa}>0$ with
the following property.  Let $\mathcal H$ be a nonzero finite-dimensional
complex Hilbert space and let
\[
 T=\begin{pmatrix}0&P\\Q&0\end{pmatrix}\in\mathcal B(\mathcal H\oplus \mathcal H),
 \qquad
 \norm{T}\le1,
\]
where one of the two blocks $P,Q$ is invertible with inverse of norm at most $\kappa$.
\begin{enumerate}[label=\textup{(\roman*)},leftmargin=2.2em]
\item There are $B,C\in\mathcal B(\mathcal H\oplus\mathcal H)$ with
$T=\comm{B}{C}$, $\norm{B},\norm{C}\le M_{\kappa}$, and
$\Gamma(B,C)\le M_{\kappa}$.
\item If $A\in\mathcal B(\mathcal H\oplus \mathcal H)$ is traceless and $\norm{A-tT}\le\eta_{\kappa}t$ for some $t>0$, then $\muc(A)\le M_{\kappa}t$.
\end{enumerate}
\end{proposition}

Part (i) controls the conditioning of a particular representation of $T$.
More explicitly, Lemma~\ref{lem:absorption} shows that its factors can be
changed by at most $4M_\kappa\|E\|$ to represent $T+E$ whenever
$\Tr E=0$ and $\|E\|\le\cabs M_\kappa^{-2}$.  Scaling this statement gives
the neighborhood estimate in part (ii).

\subsection{Quantum expanders and the trace-class Poincar\'e inequality}\label{sec:poincare}

We first prove Theorem~\ref{thm:s1-poincare}, which supplies the unitaries
used to control approximate common commutants in
Proposition~\ref{prop:graded}.

Hastings's theorem gives two unitaries for which a matrix that nearly
commutes with both must be close to a scalar in Hilbert--Schmidt norm.
To obtain the analogous statement in trace norm, we center a Hermitian
matrix at a median and apply its signed square root.  Its squared
Hilbert--Schmidt norm is the trace norm of the centered matrix, and
Ricard's theorem controls the change in its commutators.  The signed square
root need not be traceless.  The median bounds its positive and negative
spectral supports, ensuring that the scalar part contains at most half of
its squared Hilbert--Schmidt norm.  We can therefore apply the spectral gap
to the remaining nonscalar part.

We begin with the degree-four Hermitian case of Hastings's theorem
\cite{Hastings}.

\begin{theorem}[Hastings \cite{Hastings}]\label{thm:hastings}
Let $U,V$ be independent Haar-distributed unitary matrices in $\Un$, and
define the self-adjoint channel $\Phi_{U,V}:\Mn\to\Mn$ by
\[
 \Phi_{U,V}(Y)=\frac14\bigl(UYU^*+U^*YU+VYV^*+V^*YV\bigr).
\]
For every $\rho>\sqrt3/2$, the probability that the restriction of $\Phi_{U,V}$ to the Hilbert--Schmidt orthogonal complement of the scalars has operator norm at most $\rho$ tends to one as $n\to\infty$.
\end{theorem}

For later use, we record the deterministic Poincar\'e inequality furnished by
Hastings's theorem.

\begin{corollary}\label{cor:expander}
There is an absolute constant $\Kexp$ such that, for every $n$, one can choose unitary matrices $U,V\in\Un$ satisfying
\[
 \norm{Y-\tau(Y)I}_2^2
 \le \Kexp\bigl(\norm{\comm{Y}{U}}_2^2+\norm{\comm{Y}{V}}_2^2\bigr)
\]
for every $Y\in \Mn$.
\end{corollary}

\begin{proof}
Fix $\rho$ with $\sqrt3/2<\rho<1$.  Theorem~\ref{thm:hastings} shows that, for every sufficiently large $n$, there are $U,V\in\Un$ for which
\[
 \|\Phi_{U,V}|_{\{I\}^{\perp}}\|_{S_2\to S_2}\le\rho.
\]
Given $Y\in \Mn$, put $Y_0=Y-\tau(Y)I$.  Commutators do not see scalar matrices, so it is enough to estimate $Y_0$.  For a unitary matrix $U$,
\[
 \begin{aligned}
 \|[Y_0,U]\|_2^2
 &=\|Y_0-UY_0U^*\|_2^2\\
 &=2\|Y_0\|_2^2-2\operatorname{Re}\langle Y_0,UY_0U^*\rangle_2.
 \end{aligned}
\]
Adding the analogous identity for $V$ and using the definition of the self-adjoint channel $\Phi_{U,V}$ gives
\[
 \|[Y_0,U]\|_2^2+\|[Y_0,V]\|_2^2
 =4\langle Y_0,(I-\Phi_{U,V})Y_0\rangle_2.
\]
Because $Y_0\perp I$ and $\Phi_{U,V}$ is self-adjoint on Hilbert--Schmidt space,
\[
 \langle Y_0,(I-\Phi_{U,V})Y_0\rangle_2
 \ge(1-\rho)\|Y_0\|_2^2.
\]
Thus the desired inequality holds in all sufficiently large dimensions with
$\Kexp=(4(1-\rho))^{-1}$.

For each of the finitely many remaining dimensions $n\ge2$, choose an irreducible pair of unitary matrices: for instance, the cyclic shift and a diagonal unitary with pairwise distinct diagonal entries.  Their common commutant is $\mathbb CI$.  On the compact unit sphere $\mathcal S_n=\{Y:\tau(Y)=0,\ \|Y\|_2=1\}$, the continuous function $F_n(Y)=\|[Y,U]\|_2^2+\|[Y,V]\|_2^2$ is strictly positive, because $F_n(Y)=0$ would put $Y$ in the common commutant and hence force $Y=0$.  Therefore $c_n=\min_{\mathcal S_n}F_n>0$.  Taking the maximum of $c_n^{-1}$ over the finitely many exceptional dimensions and the large-dimensional constant above yields one absolute $\Kexp$.  The case $n=1$ is vacuous.
\end{proof}

\begin{remark}\label{rem:expander-sources}
For further background on quantum expanders, see Pisier \cite{Pisier}
and the explicit constructions of Ben-Aroya, Schwartz, and Ta-Shma
\cite{BST}.  The argument here uses the particular two-unitary spectral
gap supplied by Hastings.
\end{remark}

The Hilbert--Schmidt estimate must next be transferred to trace class.  Ricard's
H\"older continuity theorem for the noncommutative Mazur map provides exactly
that step.

\begin{theorem}[Ricard \cite{Ricard}]\label{thm:ricard}
For $1\le p,q<\infty$, the noncommutative Mazur map
\[
 M_{p,q}(x)=u|x|^{p/q},
 \qquad x=u|x|
\]
is $\min\{p/q,1\}$-H\"older on the unit ball of every noncommutative $L_p$-space, with a constant depending only on $p,q$.  In particular, there is an absolute $\CRic$ such that, whenever $X,Y\in \Mn$ satisfy $\norm{X}_1=\norm{Y}_1$, one has
\[
 \norm{M_{1,2}(X)-M_{1,2}(Y)}_2^2
 \le \CRic\norm{X-Y}_1.
\]
\end{theorem}

\begin{proof}[Derivation of the displayed specialization]
Let $a=\|X\|_1=\|Y\|_1$.  If $a=0$, then $X=Y=0$.  Assume $a>0$.  The matrices $a^{-1}X$ and $a^{-1}Y$ lie in the trace-norm unit ball, so Ricard's theorem with $(p,q)=(1,2)$ gives
\[
 \|M_{1,2}(a^{-1}X)-M_{1,2}(a^{-1}Y)\|_2
 \le C\|a^{-1}(X-Y)\|_1^{1/2}.
\]
The Mazur map is homogeneous of degree $1/2$:
$M_{1,2}(a^{-1}X)=a^{-1/2}M_{1,2}(X)$.  Multiplying by $a^{1/2}$ and then squaring gives
\[
 \|M_{1,2}(X)-M_{1,2}(Y)\|_2^2
 \le C^2\|X-Y\|_1.
\]
Renaming $C^2$ as $\CRic$ proves the stated form, with a constant independent of the matrix dimension.
\end{proof}

The elementary fact that a median minimizes the sum of absolute deviations
identifies a useful scalar to subtract from a Hermitian matrix.  Besides
minimizing trace norm, this choice bounds both spectral support dimensions;
that bound will control the scalar part after taking a signed square root.

\begin{lemma}\label{lem:median}
Let $Z\in\Mn$ be Hermitian.  A real number $\lambda$ satisfies
$\norm{Z-\lambda I}_1=d_1(Z)$ if and only if it is a median of the
eigenvalues of $Z$.  For such a median, the positive and negative spectral
supports of $Z-\lambda I$ each have rank at most $n/2$.
\end{lemma}

\begin{proof}
Diagonalize $Z$ and write its eigenvalues in nondecreasing order as
$\lambda_1\le\cdots\le\lambda_n$.  For a real scalar $t$,
\[
 \|Z-tI\|_1=\sum_{j=1}^n|\lambda_j-t|.
\]
The one-sided derivatives of this convex, piecewise-linear function show that it is minimized precisely when at least half of the eigenvalues lie on each side of $t$, which is the median condition.  A complex scalar cannot improve the minimum: if $\lambda=a+ib$, then
\[
 \|Z-\lambda I\|_1
 =\sum_j\sqrt{(\lambda_j-a)^2+b^2}
 \ge\sum_j|\lambda_j-a|.
\]
For a median $t$, the positive and negative spectral subspaces of $Z-tI$ each contain at most $n/2$ eigenvectors, with zero eigenvectors belonging to neither support.
\end{proof}

The expander estimate, Ricard's theorem, and the median reduction now combine
to prove the trace-class Poincar\'e inequality.

\begin{proof}[Proof of Theorem~\ref{thm:s1-poincare}]
Choose $U,V$ as in Corollary~\ref{cor:expander}.  We first prove the estimate for a self-adjoint $Z$.

\emph{Step 1: center $Z$ at a median.}
Let $\lambda$ be a median of the eigenvalues and put $X=Z-\lambda I$.  Lemma~\ref{lem:median} gives $\|X\|_1=d_1(Z)$, and the positive and negative spectral supports of $X$ both have rank at most $n/2$.  Define $M=M_{1,2}(X)=\operatorname{sgn}(X)|X|^{1/2}$. Then $M=M^*$, it has the same positive and negative supports as $X$, and $\|M\|_2^2=\Tr|X|=\|X\|_1=d_1(Z)$. 
\emph{Step 2: the scalar part of $M$ cannot contain most of its mass.}
Since $M$ is self-adjoint,
\[
 \|M-\tau(M)I\|_2^2
 =\|M\|_2^2-\frac1n|\Tr M|^2.
\]
Suppose first that $\Tr M\ge0$.  Then $|\Tr M|\le\Tr M_+$.  The support of $M_+$ has rank at most $n/2$, so Cauchy--Schwarz on that support gives
\[
 |\Tr M|^2
 \le(\rank\operatorname{supp}M_+)\,\Tr(M_+^2)
 \le\frac n2\|M\|_2^2.
\]
If $\Tr M<0$, the identical argument with $M_-$ gives the same conclusion.  Hence
\[
 \|M-\tau(M)I\|_2^2\ge\frac12\|M\|_2^2
 =\frac12d_1(Z).
\]

\emph{Step 3: transfer the expander gradient through the Mazur map.}
The Mazur map is equivariant under unitary conjugation and trace norm is unitarily invariant.  Therefore $X$ and $UXU^*$ have equal trace norm, and Theorem~\ref{thm:ricard} gives
\[
 \begin{aligned}
 \|[M,U]\|_2^2
 &=\|M-UMU^*\|_2^2\\
 &=\|M_{1,2}(X)-M_{1,2}(UXU^*)\|_2^2\\
 &\le \CRic\|X-UXU^*\|_1
 =\CRic\|[X,U]\|_1.
 \end{aligned}
\]
The same estimate holds with $V$ in place of $U$.  Applying Corollary~\ref{cor:expander} to $M$ and using $[X,U]=[Z,U]$ yields
\[
 \frac12d_1(Z)
 \le\Kexp \CRic\bigl(\|[Z,U]\|_1+\|[Z,V]\|_1\bigr).
\]
This proves the self-adjoint case.

\emph{Step 4: reduce a general matrix to two self-adjoint matrices.}
Write $Z=A+iB$ with $A=A^*$ and $B=B^*$.  If $a,b\in\mathbb R$ nearly minimize $d_1(A)$ and $d_1(B)$, then $\|Z-(a+ib)I\|_1 \le\|A-aI\|_1+\|B-bI\|_1$, so $d_1(Z)\le d_1(A)+d_1(B)$.  For any unitary $W$, $[A,W]=\tfrac12\bigl([Z,W]+[Z^*,W]\bigr)$. Moreover, $[Z^*,W]^*=W^*Z-ZW^*=W^*[Z,W]W^*$, so $\|[Z^*,W]\|_1=\|[Z,W]\|_1$.  Hence
$\|[A,W]\|_1\le\|[Z,W]\|_1$.  The same conclusion holds for
$B=(Z-Z^*)/(2i)$.  Applying the self-adjoint estimate to $A$ and $B$ gives
\[
 d_1(Z)
 \le4\Kexp \CRic\bigl(\|[Z,U]\|_1+\|[Z,V]\|_1\bigr).
\]
Thus one may take $\Kpoin=4\Kexp \CRic$.
\end{proof}

\subsection{Linearization and nonlinear absorption}\label{sec:linearization}

The trace-class inequality will control $\Gamma(B,C)$ for the factors
constructed in Section~\ref{sec:graded}.  Before choosing those factors,
we show why a bound for $\Gamma$ gives stability under perturbation.
Its definition bounds the adjoint of the linearized commutator map from
below.  The standard Hahn--Banach argument below turns that bound into a
solution estimate for the original map.  Successive corrections using this
fixed linearization then solve the nonlinear equation.  Equip $M_n^0$ with
the operator norm, and identify the dual of $\Mn$ with trace class through
$\langle E,Z\rangle=\Tr(Z^*E)$.

\begin{lemma}\label{lem:linear-surj}
Let $B,C\in \Mn$ with $\Gamma(B,C)<\infty$.  Then every $E\in M_n^0$ admits $X,Y\in \Mn$ with
\[
 \mathcal L_{B,C}(X,Y)=[X,C]+[B,Y]=E,
 \qquad
 \max\{\|X\|,\|Y\|\}\le\Gamma(B,C)\|E\|.
\]
\end{lemma}

\begin{proof}
For $n=1$, take $X=Y=0$.  We may therefore assume $n\ge2$.
Let
\[
 \mathcal X=\Mn\oplus_\infty \Mn,
 \qquad
 \|(X,Y)\|_{\mathcal X}=\max\{\|X\|,\|Y\|\},
 \qquad
 \mathcal Y=M_n^0
\]
with the operator norm on $\mathcal Y$.  Under the trace pairing
$\langle E,Z\rangle=\Tr(Z^*E)$, the annihilator of $\mathcal Y$ inside $S_1^n$ is exactly $\mathbb CI$.  Therefore $\mathcal Y^*=S_1^n/\mathbb CI$ isometrically, and the quotient norm of the class $[Z]$ is
\[
 \|[Z]\|_{\mathcal Y^*}
 =\inf_{\lambda\in\mathbb C}\|Z-\lambda I\|_1=d_1(Z).
\]
The dual of $\mathcal X$ is $S_1^n\oplus_1S_1^n$, with norm
$\|(R,S)\|=\|R\|_1+\|S\|_1$.

A cyclic trace calculation identifies the adjoint.  Indeed,
\[
 \begin{aligned}
 \Tr\bigl(Z^*[X,C]\bigr)
 &=\Tr\bigl((CZ^*-Z^*C)X\bigr)
 =\Tr\bigl([Z,C^*]^*X\bigr),\\
 \Tr\bigl(Z^*[B,Y]\bigr)
 &=\Tr\bigl((Z^*B-BZ^*)Y\bigr)
 =\Tr\bigl([B^*,Z]^*Y\bigr).
 \end{aligned}
\]
Thus
\[
 \mathcal L_{B,C}^*[Z]=([Z,C^*],[B^*,Z]).
\]
By the definition of $\Gamma(B,C)$,
\[
 \|\mathcal L_{B,C}^*[Z]\|_{\mathcal X^*}
 \ge\Gamma(B,C)^{-1}\|[Z]\|_{\mathcal Y^*}.
\]
In particular, $\mathcal L_{B,C}^*$ is injective.

Fix $E\in\mathcal Y$.  On the range of $\mathcal L_{B,C}^*$ define $F(\mathcal L_{B,C}^*[Z])=\Tr(Z^*E)$. This is well defined because the adjoint is injective.  Since $\Tr E=0$, scalar translation of $Z$ does not change the right-hand side, and therefore
\[
 \begin{aligned}
 |F(\mathcal L_{B,C}^*[Z])|
 &\le\|E\|\inf_{\lambda\in\mathbb C}\|Z-\lambda I\|_1\\
 &\le\Gamma(B,C)\|E\|\,
       \|\mathcal L_{B,C}^*[Z]\|_{\mathcal X^*}.
 \end{aligned}
\]
Hahn--Banach extends $F$ to all of $\mathcal X^*$ with the same norm.  Because the spaces are finite dimensional, the extension is evaluation at an element $(X,Y)\in\mathcal X$, and
\[
 \max\{\|X\|,\|Y\|\}\le\Gamma(B,C)\|E\|.
\]
For every $[Z]\in\mathcal Y^*$,
\[
 \langle\mathcal L_{B,C}(X,Y),[Z]\rangle
 =\langle(X,Y),\mathcal L_{B,C}^*[Z]\rangle
 =F(\mathcal L_{B,C}^*[Z])
 =\langle E,[Z]\rangle.
\]
The dual separates points of $\mathcal Y$, so $\mathcal L_{B,C}(X,Y)=E$.
\end{proof}

The preceding lemma controls each linear correction.  The remaining
error is a commutator of the corrections, so its size is quadratic.
If $\Gamma^2\|E\|$ is sufficiently small, the resulting residuals decrease
geometrically.  The next lemma makes this familiar perturbation argument
quantitative without requiring a linear choice of the solutions.

\begin{lemma}\label{lem:absorption}
There is an absolute constant $\cabs>0$ with the following property.
Let $B,C,E\in\Mn$ and $\Gamma>0$ satisfy $\Tr E=0$ and
$\Gamma(B,C)\le\Gamma$.  If
\[
 \norm{E}\le \cabs\Gamma^{-2},
\]
then there are $X,Y\in\Mn$ satisfying
\[
 \comm{B+X}{C+Y}=\comm{B}{C}+E
\]
and
\[
 \max\{\norm{X},\norm{Y}\}\le4\Gamma\norm{E}.
\]
\end{lemma}

\begin{proof}
Put $e=\|E\|$.  The case $e=0$ is trivial.  We construct corrections iteratively.  Set $X_0=Y_0=0$, $R_0=E$. At stage $k$, suppose $R_k=E-\mathcal L_{B,C}(X_k,Y_k)-[X_k,Y_k]$. The residual is traceless, because every other term in this identity is.  By Lemma~\ref{lem:linear-surj}, choose $x_k,y_k$ with
\[
 \mathcal L_{B,C}(x_k,y_k)=R_k,
 \qquad
 q_k:=\max\{\|x_k\|,\|y_k\|\}
 \le\Gamma\|R_k\|.
\]  Put
$X_{k+1}=X_k+x_k$ and $Y_{k+1}=Y_k+y_k$.  Expanding the quadratic term gives
\[
 R_{k+1}
 =-[X_k,y_k]-[x_k,Y_k]-[x_k,y_k].
\]

We prove simultaneously that
\[
 \|R_k\|\le2^{-k}e
 \quad\text{and}\quad
 \max\{\|X_k\|,\|Y_k\|\}\le4\Gamma e
\]
for every $k$.  The assertion is clear at $k=0$.  Assuming it at stage $k$, the commutator inequality
$\|[R,S]\|\le2\|R\|\,\|S\|$ yields
\[
 \begin{aligned}
 \|R_{k+1}\|
 &\le2\|X_k\|\,\|y_k\|+2\|x_k\|\,\|Y_k\|
       +2\|x_k\|\,\|y_k\|\\
 &\le16\Gamma e q_k+2q_k^2\\
 &\le16\Gamma^2e\|R_k\|+2\Gamma^2\|R_k\|^2\\
 &\le18\Gamma^2e\|R_k\|.
 \end{aligned}
\]
Choose $\cabs\le1/36$.  The hypothesis $e\le \cabs\Gamma^{-2}$ then gives
$\|R_{k+1}\|\le\frac12\|R_k\|\le2^{-(k+1)}e$.  Moreover,
\[
 \sum_{j=0}^kq_j
 \le\Gamma e\sum_{j=0}^k2^{-j}
 <2\Gamma e,
\]
which proves the required bound for $X_{k+1}$ and $Y_{k+1}$.

The series $\sum x_k$ and $\sum y_k$ converge in operator norm to matrices $X$ and $Y$.  Their norms satisfy $\max\{\|X\|,\|Y\|\}\le4\Gamma e$. Since $R_k\to0$, continuity of the commutator and of $\mathcal L_{B,C}$ in the identity defining $R_k$ gives
\[
 E=\mathcal L_{B,C}(X,Y)+[X,Y].
\]
Equivalently, $[B+X,C+Y]=[B,C]+E$.
\end{proof}

\subsection{Graded matrices with an invertible link}\label{sec:graded}

We now combine the trace-class inequality from Section~\ref{sec:poincare}
with the correction argument from Section~\ref{sec:linearization} to prove
Proposition~\ref{prop:graded}.  We must choose factors for which
$\Gamma(B,C)$ is uniformly bounded.  The grading operator controls off-diagonal
blocks in the dual estimate.  The invertible link makes the two diagonal
blocks close after a suitable similarity, leaving one matrix to control.
The expander estimate forces this remaining matrix to be close to a scalar.
The corresponding block-diagonal addition to the second factor leaves the
commutator unchanged because it commutes with the grading operator.  Once
this construction bounds $\Gamma$, Lemma~\ref{lem:absorption} gives the
neighborhood estimate.

\begin{proof}[Proof of Proposition~\ref{prop:graded}]
We treat the case in which $Q$ is invertible.  Let
\[
 G=\diag(I,-I),
 \qquad
 C_0=\frac12GT
 =\frac12\begin{pmatrix}0&P\\-Q&0\end{pmatrix}.
\]
A direct multiplication gives $[G,C_0]=T$.  Put $S=Q^*$.  The assumptions imply
$\|S\|\le1$ and $\|S^{-1}\|\le\kappa$.  Choose unitary matrices $U,W$ on the first block satisfying Theorem~\ref{thm:s1-poincare}, and define $V=S^{-1}WS$, $B=G$, $C=C_0+\diag(U^*,V^*)$. The diagonal summand commutes with $G$, so $[B,C]=T$.  Also
$\|V\|\le\|S^{-1}\|\,\|S\|\le\kappa$, and hence $\|B\|=1$, $\|C\|\le\|C_0\|+1+\kappa\le C(1+\kappa)$. It remains to prove a dimension-free bound for $\Gamma(B,C)$.

\emph{Step 1: block-diagonal dual witnesses.}
Consider first a block-diagonal witness $Z_0=\diag(Z_{11},Z_{22})$ and put $R_0=[Z_0,C^*]$.  Since
\[
 C_0^*=\frac12\begin{pmatrix}0&-S\\P^*&0\end{pmatrix},
\]
block multiplication gives
\[
 R_0=
 \begin{pmatrix}
 [Z_{11},U]&\frac12(SZ_{22}-Z_{11}S)\\[1mm]
 \frac12(Z_{22}P^*-P^*Z_{11})&[Z_{22},V]
 \end{pmatrix}.
\]
Taking the block-diagonal conditional expectation and the upper-right corner, and using Lemma~\ref{lem:block-facts}, yields
\[
 \|[Z_{11},U]\|_1+\|[Z_{22},V]\|_1\le\|R_0\|_1,
 \qquad
 \|SZ_{22}-Z_{11}S\|_1\le2\|R_0\|_1.
\]
Define $Z'=SZ_{22}S^{-1}$.  The identity $Z'-Z_{11}=(SZ_{22}-Z_{11}S)S^{-1}$ gives $\|Z'-Z_{11}\|_1\le2\kappa\|R_0\|_1$. Furthermore, because $V=S^{-1}WS$, $[Z',W]=S[Z_{22},V]S^{-1}$, and therefore
\[
 \|[Z',W]\|_1
 \le\|S\|\,\|S^{-1}\|\,\|[Z_{22},V]\|_1
 \le\kappa\|R_0\|_1.
\]
Since $W$ is unitary,
\[
 \begin{aligned}
 \|[Z_{11},W]\|_1
 &\le\|[Z',W]\|_1+\|[Z_{11}-Z',W]\|_1\\
 &\le\kappa\|R_0\|_1+2\|Z_{11}-Z'\|_1\\
 &\le5\kappa\|R_0\|_1.
 \end{aligned}
\]
Theorem~\ref{thm:s1-poincare}, applied to the expander pair $U,W$, now shows $d_1(Z_{11})\le C\kappa\|R_0\|_1$. Choose $\lambda\in\mathbb C$ with
$\|Z_{11}-\lambda I\|_1\le2d_1(Z_{11})$.  Similarity by $S$ costs at most
$\|S\|\,\|S^{-1}\|\le\kappa$ in trace norm, so
\[
 \begin{aligned}
 \|Z_{22}-\lambda I\|_1
 &=\|S^{-1}(Z'-\lambda I)S\|_1\\
 &\le\kappa\bigl(\|Z'-Z_{11}\|_1+\|Z_{11}-\lambda I\|_1\bigr)\\
 &\le C_\kappa\|R_0\|_1.
 \end{aligned}
\]
The trace norm of a block-diagonal matrix is the sum of the trace norms of its diagonal blocks.  Hence
\[
 d_1(Z_0)
 \le\|Z_{11}-\lambda I\|_1+\|Z_{22}-\lambda I\|_1
 \le C_\kappa\|[Z_0,C^*]\|_1.
\]

\emph{Step 2: arbitrary dual witnesses.}
Write
\[
 Z=\begin{pmatrix}Z_{11}&Z_{12}\\Z_{21}&Z_{22}\end{pmatrix},
 \qquad
 Z_0=\diag(Z_{11},Z_{22}),
 \qquad
 R=Z-Z_0.
\]
Because $G=\diag(I,-I)$,
\[
 [G,Z]=2\begin{pmatrix}0&Z_{12}\\-Z_{21}&0\end{pmatrix}.
\]
Lemma~\ref{lem:block-facts}(iii) therefore gives the exact trace-norm identity $\|R\|_1=\tfrac12\|[G,Z]\|_1$. Also
\[
 \|[Z_0,C^*]\|_1
 \le\|[Z,C^*]\|_1+\|[R,C^*]\|_1
 \le\|[Z,C^*]\|_1+2\|C\|\,\|R\|_1.
\]
For every scalar $\lambda$,
$\|Z-\lambda I\|_1\le\|Z_0-\lambda I\|_1+\|R\|_1$, so
$d_1(Z)\le d_1(Z_0)+\|R\|_1$.  Combining these estimates with Step 1 gives
\[
 d_1(Z)
 \le C_\kappa\bigl(\|[Z,C^*]\|_1+\|[G,Z]\|_1\bigr).
\]
Since $B=G=G^*$, this is exactly the assertion $\Gamma(B,C)\le M_\kappa$ after enlarging the constant.

\emph{Step 3: absorb a relative perturbation.}
Represent $tT$ by $B_t=\sqrt t\,B$, $C_t=\sqrt t\,C$. The numerator in the definition of $\Gamma$ is unchanged, whereas both commutators in the denominator are multiplied by $\sqrt t$.  Thus
\[
 \Gamma(B_t,C_t)=t^{-1/2}\Gamma(B,C)
 \le M_\kappa t^{-1/2}.
\]
Let $E=A-tT$ and choose
$\eta_\kappa\le \cabs{}M_\kappa^{-2}$.  If
$\|E\|\le\eta_\kappa t$, then $\|E\|\le \cabs\Gamma(B_t,C_t)^{-2}$, so Lemma~\ref{lem:absorption} supplies $X,Y$ with
$A=[B_t+X,C_t+Y]$ and
\[
 \max\{\|X\|,\|Y\|\}
 \le4M_\kappa t^{-1/2}\|E\|
 \le4M_\kappa\eta_\kappa\sqrt t.
\]
The unperturbed factors have norm at most $M_\kappa\sqrt t$; hence both perturbed factors have norm at most $C_\kappa\sqrt t$, and their product is at most $M_\kappa t$ after one final enlargement of the constant.

If $P$ rather than $Q$ is invertible, conjugate by the unitary matrix that interchanges the two summands of the grading.  This preserves all relevant norms and reduces to the case just proved.
\end{proof}
\section{Invertible links and trace gaps}\label{sec:links}

Proposition~\ref{prop:graded} gives a commutator bound near graded matrices
with an invertible link.  We now extend that bound to matrices whose
diagonal blocks need not be small.  First, a similarity and paving produce
compressions that lie in the neighborhoods covered by the proposition.
The gluing estimate from Section~\ref{sec:gluing} then gives a bound for
the full matrix whenever it has an invertible off-diagonal block with a
controlled inverse.

We next replace this invertibility hypothesis with a substantial trace
imbalance between two subspaces of equal dimension.  Mixing coordinates
and paving turn the trace imbalance into invertible links on finitely many
compressions.  This trace condition is what the final argument will use:
Section~\ref{sec:rank} obtains it from a lower bound on spectral mass in
even dimension.  We state the two results before proving them.

\begin{theorem}[Invertible-link theorem]\label{thm:full-link}
For every $\tau>0$ there is $C_\tau<\infty$ with the following property.  Let $\mathcal H$ be a nonzero finite-dimensional complex Hilbert space and let
\[
 A=\begin{pmatrix}A_0&X\\Y&D\end{pmatrix}
 \in\mathcal B(\mathcal H\oplus \mathcal H)
\]
be traceless, with upper-right block satisfying $s_{\min}(X)\ge\tau\norm{A}$.  Then $\muc(A)\le C_\tau\norm{A}$.
\end{theorem}

The equal block dimensions make $X$ square.  Thus, after normalizing
$\|A\|=1$, the hypothesis $s_{\min}(X)\ge\tau$ says that $X$ is invertible
with $\|X^{-1}\|\le\tau^{-1}$.  The next theorem replaces this explicit
invertibility assumption with a trace condition.

\begin{theorem}[Balanced trace-gap theorem]\label{thm:trace-gap}
For every $\delta>0$ there is $C_{\delta}<\infty$ with the following property.  Let $\mathcal H$ be a nonzero finite-dimensional complex Hilbert space and let
\(A=\left(\begin{smallmatrix}A_0&X\\Y&D\end{smallmatrix}\right)
\in\mathcal B(\mathcal H\oplus \mathcal H)\)
be traceless.  If $|\Tr A_0|\ge\delta\,\dim(\mathcal H)\,\norm{A}$, then
$\muc(A)\le C_{\delta}\norm{A}$.
\end{theorem}

Since $A$ is traceless and the two blocks have the same dimension, the lower-right block then has normalized trace of the same size and opposite sign.  The theorem says that such a trace imbalance between two halves already forces a dimension-free commutator bound.

\subsection{Paving, shearing, and proof of the invertible-link theorem}

We first treat a zero upper-left block.  After making the invertible
link positive and the remaining diagonal block hollow, paving makes each
compression close to a graded matrix.  Positivity ensures that the link
stays invertible on every compression.  We choose the paving accuracy to
enter the neighborhood in Proposition~\ref{prop:graded}; the number of
blocks then depends only on the lower bound for the link.  Finite-block
gluing recovers the full matrix.  A triangular similarity reduces the
general case to this zero-corner case.

The paving input is Ravichandran and Srivastava's multi-paving theorem
\cite{RS}, stated here in the one-sided Hermitian form that we use.

\begin{theorem}[Ravichandran--Srivastava \cite{RS}]\label{thm:paving-rs}
Let $k\ge1$ be an integer and let $0<\eps<1$.  Given zero-diagonal Hermitian contractions $A^{(1)},\ldots,A^{(k)}\in \Mn$, there exists a partition $\{S_1,\ldots,S_r\}$ of $[n]$ where $r\le18k\eps^{-2}$ such that
\[
 \lambda_{\max}\bigl(P_{S_j}A^{(l)}P_{S_j}\bigr)<\eps
 \quad\text{for }j\in [r]\text{ and }l \in [k].
\]
\end{theorem}

Applying the Ravichandran--Srivastava theorem to
$\pm\operatorname{Re}T$ and $\pm\operatorname{Im}T$ gives the following
single-matrix consequence.  We record the elementary specialization with
the number of classes prescribed.

\begin{corollary}\label{thm:paving}
Given a zero-diagonal matrix $T\in \Mn$ and an integer $r\ge1$, there exists a partition $\{S_1,\ldots,S_r\}$ of $[n]$ such that
\[
 \norm{P_{S_j}TP_{S_j}}\le \frac{18}{\sqrt r}\norm{T}
 \quad\text{for }j\in [r].
\]
\end{corollary}

\begin{proof}
The case $T=0$ is trivial, so assume $\norm{T}=1$ by homogeneity.  If $r\le 81$, then $18/\sqrt r\ge 2$, and any partition works, because every compression of $T$ has norm at most $\norm{T}$.

Let $r>81$ and put $\eps=9/\sqrt r$, so that $0<\eps<1$ and $72\eps^{-2}<r$.  The four matrices $\pm\operatorname{Re}T$ and $\pm\operatorname{Im}T$ are Hermitian contractions with zero diagonal, so Theorem~\ref{thm:paving-rs} with $k=4$ gives a partition $\{S_1,\ldots,S_{r'}\}$ of $[n]$ with $r'\le 72\eps^{-2}<r$ on which all four have largest eigenvalue less than $\eps$.  Bounding a Hermitian matrix and its negative bounds its norm, so $\norm{P_{S_j}(\operatorname{Re}T)P_{S_j}}\le\eps$ and $\norm{P_{S_j}(\operatorname{Im}T)P_{S_j}}\le\eps$ for every $j$, whence $\norm{P_{S_j}TP_{S_j}}\le2\eps=18/\sqrt r$.
\end{proof}

We first combine this paving consequence with the graded regularity result to
handle matrices whose upper-left block is zero.

\begin{lemma}\label{lem:graded-corner}
For every $\tau>0$ there is $C'_\tau<\infty$ with the following property.  Let $\mathcal H$ be a nonzero finite-dimensional complex Hilbert space and let
\[
 A=\begin{pmatrix}0&X\\Y&D\end{pmatrix}
 \in\mathcal B(\mathcal H\oplus \mathcal H)
\]
satisfy $s_{\min}(X)\ge\tau\norm{A}$, with $D$ traceless.  Then $\muc(A)\le C'_\tau\norm{A}$.
\end{lemma}

\begin{proof}
The case $A=0$ is trivial, and by homogeneity (Lemma~\ref{lem:basic-mu}(i)) we may assume $\|A\|=1$, so that $s_{\min}(X)\ge\tau$.  We first normalize the upper-right block.  Write the polar decomposition
$X=U|X|$.  Conjugating $A$ by the block-diagonal unitary matrix $\diag(U,I)$ changes the upper-right block to $U^*X=|X|$ and preserves the hypotheses and $\muc(A)$.  Thus we may assume $X=X^*\succeq\tau I$. By Theorem~\ref{thm:hollow}(i), choose a unitary $W$ on $\mathcal H$ such that $W^*DW$ is hollow.  Conjugating both copies of $\mathcal H$ by $W$ changes $X$ to $W^*XW$, which still satisfies $W^*XW\succeq\tau I$.  We may therefore assume simultaneously that $X\succeq\tau I$ and $\diag D=0$. 
Let $\eta_{2/\tau}$ be the neighborhood radius in Proposition~\ref{prop:graded}.  Choose
\[
 0<\varepsilon\le\min\left\{1,\frac{\tau}{2},
 \frac{\tau\eta_{2/\tau}}2\right\}.
\]
Since $D$ is a compression of the contraction $A$, $\|D\|\le1$. Put $r=\lceil 324 \eps^{-2}\rceil$.  Corollary~\ref{thm:paving} partitions the index set of $\mathcal H$ into sets $S_1,\ldots,S_r$. Write $P_j=P_{S_j}$, so that $\sum_jP_j=I$ and $\|P_jDP_j\|\le\varepsilon$ for every $j$.

Use the paired orthogonal decomposition
\[
 (P_1\mathcal H\oplus P_1\mathcal H)\oplus\cdots\oplus(P_r\mathcal H\oplus P_r\mathcal H).
\]
The $j$th diagonal compression is
\[
 A_j=\begin{pmatrix}
 0&P_jXP_j\\ P_jYP_j&P_jDP_j
 \end{pmatrix}.
\]
Because $D$ is hollow, $\Tr(P_jDP_j)=0$, so $\Tr A_j=0$.  Separate its graded part by setting
\[
 T_j=\begin{pmatrix}
 0&P_jXP_j\\ P_jYP_j&0
 \end{pmatrix},
 \qquad
 t_j=\|T_j\|.
\]
Lemma~\ref{lem:block-facts}(iii) gives $t_j=\max\{\|P_jXP_j\|,\|P_jYP_j\|\}$. On the range of $P_j$, positivity of $X$ implies
$P_jXP_j\succeq\tau P_j$.  Hence this compression is invertible on $P_j\mathcal H$ and $\|(P_jXP_j)^{-1}\|\le\tau^{-1}$, $t_j\ge\tau$. Also $T_j=A_j-\diag(0,P_jDP_j)$, and $A_j$ is a compression of a contraction, so
\[
 t_j\le\|A_j\|+\|P_jDP_j\|\le1+\varepsilon\le2.
\]
Thus $t_j^{-1}T_j$ has norm one, and its upper-right block has inverse norm at most $t_j\|(P_jXP_j)^{-1}\| \le\tfrac{2}{\tau}$. Moreover,
\[
 \left\|\frac{A_j}{t_j}-\frac{T_j}{t_j}\right\|
 =\frac{\|P_jDP_j\|}{t_j}
 \le\frac{\varepsilon}{\tau}
 \le\eta_{2/\tau}.
\]
Proposition~\ref{prop:graded}, using its version with the upper-right block invertible, therefore yields $\muc(A_j/t_j)\le M_{2/\tau}$. By homogeneity, $\muc(A_j)\le M_{2/\tau}t_j\le2M_{2/\tau}$. 
The number of paired blocks depends only on $\tau$, because $r\le C\varepsilon(\tau)^{-2}$.  Lemma~\ref{lem:finite-gluing} now gives
\[
 \muc(A)
 \le \Cgl\sqrt r\left(\|A\|+\max_j\muc(A_j)\right)
 \le C'_\tau.
\]
All preliminary conjugations were unitary, so this is the desired estimate for the original matrix.
\end{proof}

To pass from a general block matrix to the preceding zero-corner form, we use
the following elementary shear.

\begin{lemma}\label{lem:shear}
Let $\mathcal H$ be a finite-dimensional complex Hilbert space and let
$A=\begin{pmatrix}A_0&X\\Y&D\end{pmatrix}\in\mathcal B(\mathcal H\oplus\mathcal H)$
with $X$ invertible, and put
$S=\begin{pmatrix}I&0\\-X^{-1}A_0&I\end{pmatrix}\in\mathcal B(\mathcal H\oplus\mathcal H)$.
Then
\[
 S^{-1}AS=
 \begin{pmatrix}
 0&X\\
 Y-DX^{-1}A_0&D+X^{-1}A_0X
 \end{pmatrix}.
\]
If $A$ is traceless, so is the lower-right block of $S^{-1}AS$.
\end{lemma}

\begin{proof}
The inverse of $S$ is
\[
 S^{-1}=\begin{pmatrix}I&0\\X^{-1}A_0&I\end{pmatrix}.
\]
Multiplying first on the right gives
\[
 AS=
 \begin{pmatrix}A_0&X\\Y&D\end{pmatrix}
 \begin{pmatrix}I&0\\-X^{-1}A_0&I\end{pmatrix}
 =\begin{pmatrix}0&X\\Y-DX^{-1}A_0&D\end{pmatrix}.
\]
Multiplying this matrix on the left by $S^{-1}$ yields
\[
 S^{-1}AS=
 \begin{pmatrix}
 0&X\\
 Y-DX^{-1}A_0&D+X^{-1}A_0X
 \end{pmatrix}.
\]
Finally, cyclicity of trace gives $\Tr(D+X^{-1}A_0X)=\Tr D+\Tr A_0=\Tr A$. Thus the new lower-right block is traceless whenever $A$ is.
\end{proof}

The zero-corner estimate and the shear now give the full invertible-link
theorem.

\begin{proof}[Proof of Theorem~\ref{thm:full-link}]
The case $A=0$ is trivial, and by homogeneity we may assume $\|A\|=1$, so that $s_{\min}(X)\ge\tau$.  Apply Lemma~\ref{lem:shear} and write $A'=S^{-1}AS$.  Since
$\|A_0\|\le\|A\|\le1$ and $\|X^{-1}\|\le\tau^{-1}$, $\|X^{-1}A_0\|\le\tau^{-1}$. The triangular formulas for $S$ and $S^{-1}$ therefore give $\|S\|,\|S^{-1}\|\le1+\tau^{-1}$. Set $L_\tau=(1+\tau^{-1})^2$. Then
\[
 \|A'\|\le\|S^{-1}\|\,\|A\|\,\|S\|\le L_\tau.
\]
By the shear formula, the upper-left block of $A'$ is zero, the upper-right block remains $X$, and the lower-right block is traceless.  Since $\|A'\|\le L_\tau$, the upper-right block satisfies $s_{\min}(X)\ge\tau\ge(\tau/L_\tau)\|A'\|$.  Lemma~\ref{lem:graded-corner}, applied with parameter $\tau/L_\tau$, therefore gives
\[
 \muc(A')
 \le C'_{\tau/L_\tau}\|A'\|
 \le L_\tau C'_{\tau/L_\tau}.
\]
Finally, $A=SA'S^{-1}$.  Lemma~\ref{lem:basic-mu}(iii) yields
\[
 \muc(A)
 \le(\|S\|\,\|S^{-1}\|)^2\muc(A')
 \le L_\tau^2\,L_\tau C'_{\tau/L_\tau}.
\]
The right-hand side depends only on $\tau$, and may be denoted by $C_\tau$.  Undoing the normalization gives $\muc(A)\le C_\tau\|A\|$.
\end{proof}

\subsection{From a balanced trace gap to an invertible link}

Theorem~\ref{thm:full-link} has reduced the task to finding invertible
links on a fixed number of traceless compressions.  We now construct these
links from the trace gap.  The trace condition gives an average, while
invertibility requires a lower bound in every direction.  First make the
diagonals of the two blocks equal to their respective normalized traces.
Mixing paired coordinates then puts the trace gap into the diagonal of an
off-diagonal block.  We choose phases to prevent cancellation with the
original off-diagonal entries.  Paving makes the hollow remainder small,
and the paired construction keeps every resulting compression traceless.
Theorem~\ref{thm:full-link} and finite-block gluing finish the argument.

\begin{proof}[Proof of Theorem~\ref{thm:trace-gap}]
The case $A=0$ is trivial, and by homogeneity we may assume $\|A\|=1$, so that $|\alpha|\ge\delta$ for the normalized trace $\alpha$ of $A_0$.  Since $A$ is traceless and the diagonal blocks have the same dimension, the normalized trace of $D$ is $-\alpha$.  Apply Theorem~\ref{thm:hollow}(i) separately to
$A_0-\alpha I_{\mathcal H}$ and $D+\alpha I_{\mathcal H}$.  After conjugating by a block-diagonal unitary matrix, choose an orthonormal basis of each copy of $\mathcal H$, indexed by the same finite set $\mathcal I$, in which
\[
 (A_0)_{ii}=\alpha,
 \qquad
 D_{ii}=-\alpha
 \quad\text{for }i\in\mathcal I.
\]
This operation does not alter the hypotheses or the commutator cost.

Let $Z=\diag(z_i)_{i\in\mathcal I}$ with $|z_i|=1$ and define
\[
 W_Z=\frac1{\sqrt2}
 \begin{pmatrix}I&Z\\-Z^*&I\end{pmatrix}.
\]
Because $ZZ^*=Z^*Z=I$,
\[
 W_Z^*W_Z
 =\frac12
 \begin{pmatrix}I&-Z\\Z^*&I\end{pmatrix}
 \begin{pmatrix}I&Z\\-Z^*&I\end{pmatrix}
 =I,
\]
so $W_Z$ is unitary.  Write
\[
 W_Z^*AW_Z=
 \begin{pmatrix}\widetilde A_0&\widetilde X\\\widetilde Y&\widetilde D\end{pmatrix}.
\]
Multiplying the three block matrices shows that
\[
 \widetilde X=\frac12\bigl(A_0Z+X-ZYZ-ZD\bigr).
\]
If $x_i=X_{ii}$ and $y_i=Y_{ii}$, its $i$th diagonal entry satisfies
\[
 2\widetilde X_{ii}=2\alpha z_i+x_i-y_i z_i^2.
\]
For fixed $i$, define
\[
 g_i(z)=2\alpha z+x_i-y_i z^2
 \quad\text{for }z\in\mathbb T.
\]
The functions $1,z,z^2$ are orthogonal in $L_2(\mathbb T)$, so
\[
 \int_{\mathbb T}|g_i(z)|^2\,dm(z)
 =4|\alpha|^2+|x_i|^2+|y_i|^2
 \ge4|\alpha|^2.
\]
The supremum of $|g_i|$ is at least its $L_2$ norm.  We may therefore choose each phase $z_i$ independently so that
$|g_i(z_i)|\ge2|\alpha|$.  With this choice,
\[
 |\widetilde X_{ii}|\ge|\alpha|\ge\delta
 \quad\text{for }i\in\mathcal I.
\]

Let $G=\diag(\widetilde X_{ii})_{i\in\mathcal I}$ and $R=\widetilde X-G$.  Then $R$ is hollow and
$s_{\min}(G)=\min_i|\widetilde X_{ii}|\ge\delta$.  Since $\widetilde X$ is a corner of the contraction $W_Z^*AW_Z$, $\|\widetilde X\|\le1$; also $\|G\|\le1$, and hence $\|R\|\le2$.
Put $r=\lceil 5184\delta^{-2}\rceil$ so that $\tfrac{36}{\sqrt r}\le\tfrac\delta2$. Since $\|R\|\le 2$, Corollary~\ref{thm:paving} supplies a partition
$\mathcal I=S_1\sqcup\cdots\sqcup S_r$ for which $\|P_{S_j}RP_{S_j}\|\le \frac{18}{\sqrt{r}}\|R\|\le \tfrac\delta2$. For every unit vector $v$ in $P_{S_j}\mathcal H$,
\[
 \|(P_{S_j}GP_{S_j}+P_{S_j}RP_{S_j})v\|
 \ge\|P_{S_j}GP_{S_j}v\|-\|P_{S_j}RP_{S_j}\|.
\]
Taking the infimum over $v$ gives
\[
 s_{\min}(P_{S_j}\widetilde XP_{S_j})
 \ge s_{\min}(P_{S_j}GP_{S_j})-
       \|P_{S_j}RP_{S_j}\|
 \ge\frac\delta2.
\]

Let $E_j=P_{S_j}\oplus P_{S_j}$.  Since both $P_{S_j}$ and $Z$ are diagonal, $E_j$ commutes with $W_Z$.  Consequently,
\[
 \begin{aligned}
 \Tr(E_jW_Z^*AW_ZE_j)
 &=\Tr(E_jW_Z^*AW_Z)\\
 &=\Tr(W_ZE_jW_Z^*A)\\
 &=\Tr(E_jA)\\
 &=\Tr(P_{S_j}A_0)+\Tr(P_{S_j}D)\\
 &=|S_j|(\alpha-\alpha)=0.
 \end{aligned}
\]
Thus every diagonal compression $E_jW_Z^*AW_ZE_j$ belongs to
$\mathcal B(P_{S_j}\mathcal H\oplus P_{S_j}\mathcal H)$, is traceless with norm at most one, and has an upper-right block whose least singular value is at least $\delta/2$.  Theorem~\ref{thm:full-link} bounds its commutator cost by $C_{\delta/2}$.  The number $r$ depends only on $\delta$, so Lemma~\ref{lem:finite-gluing} gives
\[
 \muc(W_Z^*AW_Z)
 \le \Cgl\sqrt r\bigl(1+C_{\delta/2}\bigr)
 =:C_\delta.
\]
Unitary invariance of $\muc$ completes the proof.
\end{proof}
\section{Approximate-rank rigidity}\label{sec:rank}

The preceding section gives a bound whenever an even-dimensional matrix
has a substantial trace gap.  We now show why a possible counterexample
must have enough spectral mass to produce such a gap.  If a matrix is
close to one of small rank, simultaneous hollowization and vector
partitioning give smaller traceless compressions whose norm reduction
compensates for the gluing cost from Section~\ref{sec:gluing}.  Under the
bound in smaller dimensions, such a matrix cannot have large commutator
cost.  A possible counterexample must therefore have large approximate
rank.  The resulting lower bound on its Hilbert--Schmidt norm yields, in
even dimension, the trace imbalance required by
Theorem~\ref{thm:trace-gap}.  This is the global dichotomy used in the final
proof.

The low-rank argument combines Damm--Fa\ss bender simultaneous
hollowization \cite{DammFassbender} with
Marcus--Spielman--Srivastava vector partitioning \cite{MSS}.  We apply
vector partitioning to the support projection of a low-rank approximation.
Small normalized rank gives compression norms of order
$1/r$ once the parameters are fixed appropriately, enough to overcome the
$\sqrt r$ gluing cost.  This is where the approximate-rank hypothesis
improves on the general paving estimate.

For $A\in\Mn$ and $\eps>0$, define
\[
 \rank_\eps(A)=\min\{\rank F:F\in\Mn,\ \norm{A-F}\le\eps\norm{A}\}.
\]
Thus $\rank_\eps(A)$ is the smallest rank of a matrix approximating $A$ to
relative error at most $\eps$ in operator norm.

\begin{theorem}[Approximate-rank exclusion]\label{thm:rank-exclusion}
There are absolute positive constants $\eps_0,\delta_0$ and $K_0$ with the following property.  Let $K\ge K_0$ and let $A\in \Mn$ be traceless.  Assume
\begin{enumerate}[label=\textup{(\alph*)},leftmargin=2.2em]
\item $\muc(D)\le K\norm{D}$ for every traceless $D\in M_m(\mathbb C)$ with $1\le m<n$, and
\item $\muc(A)>\tfrac34K\norm{A}$.
\end{enumerate}
Then $\rank_{\eps_0}(A)>\delta_0n$.  Consequently,
$\norm{A}_2^2\ge c_0n\norm{A}^2$ where $c_0=\tfrac12\delta_0\eps_0^2$.
\end{theorem}

The final proof will combine the exclusion theorem with the following
consequence of the balanced trace-gap theorem.  The exclusion theorem
uses the bound in smaller dimensions to force large stable rank.  The
corollary requires no inductive hypothesis: in even dimension, a stable
rank $\|A\|_2^2/\|A\|^2$ proportional to the dimension already gives a
uniform commutator bound.

\begin{corollary}\label{cor:stable-easy}
Let $\sigma>0$ and let $A\in \Mn$ be traceless of even dimension.  If
$\norm{A}_2^2\ge\sigma n\norm{A}^2$,
then
\[
 \muc(A)\le C_{\sigma/4}\norm{A},
\]
where $C_\delta$ is the constant of Theorem~\textup{\ref{thm:trace-gap}}.
\end{corollary}

\subsection{Vector partitioning and proof of approximate-rank rigidity}

To prove Theorem~\ref{thm:rank-exclusion}, we assume that a low-rank
approximation exists and construct a partition whose gluing bound
contradicts the assumed lower bound on $\muc(A)$.
Write $A=F+R$, where $F$ has low rank and $R$ has small norm, and let
$P$ project onto $\ran F+\ran F^*$.  The reason to partition this support
projection is the identity $F=PFP$: for every coordinate projection $P_S$,
\[
 \|P_SFP_S\|\le\|F\|\,\|P_SP\|^2
 =\|F\|\,\|P_SPP_S\|.
\]
Thus a small compression of $P$ gives a small compression of $F$ with the
same norm reduction.  Damm and Fa\ss bender's theorem makes $A$ hollow
while spreading the diagonal of $P$ evenly outside two coordinates.
The vector partition theorem below then controls the compressions of $P$.

The gain is stronger than general paving provides: when the diagonal of
$P$ is sufficiently small, the bound is close to $1/r$.  This offsets the
$\sqrt r$ cost of gluing.  The two exceptional coordinates are isolated as
singleton blocks, on which the hollow matrix $A$ vanishes.  The precise
partition theorem is Corollary~1.5 of Marcus, Spielman, and Srivastava
\cite{MSS}.

\begin{theorem}[Marcus--Spielman--Srivastava \cite{MSS}]\label{thm:mss}
Let $r$ be a positive integer, let $\delta>0$, and let $u_1,\ldots,u_m\in\C^d$ be vectors such that
\[
 \sum_{i=1}^m u_i u_i^*=I_d,
\]
and $\norm{u_i}^2\le\delta$ for all $i$.
Then there exists a partition
$\{S_1,\ldots,S_r\}$ of $[m]$
such that
\[
 \left\|\sum_{i\in S_j}u_i u_i^*\right\|
 \le\left(\frac1{\sqrt r}+\sqrt\delta\right)^2
 \quad \text{for } j=1,\ldots,r.
\]
\end{theorem}

Deleting the exceptional coordinates gives a family whose rank-one
operators sum to at most the identity, while the partition theorem assumes
equality.  The following elementary completion supplies that equality by
splitting the spectral decomposition of the deficit into small pieces.

\begin{lemma}\label{lem:frame-tools}
Let $\delta>0$ and let $u_1,\ldots,u_m\in\C^d$ be vectors such that
$\sum_{i=1}^m u_i u_i^*\preceq I_d$ and $\|u_i\|^2\le\delta$ for
$i\in [m]$.  Then there are $q=d\lceil\delta^{-1}\rceil$ vectors
$v_1,\ldots,v_q\in\C^d$ with $\|v_j\|^2\le\delta$ for all $j\in [q]$ such
that
\[
 \sum_{i=1}^mu_iu_i^*+\sum_{j=1}^qv_jv_j^*=I_d.
\]
\end{lemma}

\begin{proof}
Put $R=I_d-\sum_{i=1}^mu_iu_i^*$.  Then $0\preceq R\preceq I_d$.  Let
$f_1,\ldots,f_d$ be an orthonormal basis of $\C^d$ consisting of eigenvectors
of $R$, and write $R=\sum_{k=1}^d\lambda_kf_kf_k^*$ with
$0\le\lambda_k\le 1$.  Put $L=\lceil\delta^{-1}\rceil$, and for $k\in[d]$
and $\ell\in[L]$ set $v_{k,\ell}=\sqrt{\lambda_k/L}\,f_k$.  Then
$\|v_{k,\ell}\|^2=\lambda_k/L\le\delta$ and
$\sum_{\ell=1}^Lv_{k,\ell}v_{k,\ell}^*=\lambda_kf_kf_k^*$.  Summing over
$k$ gives $\sum_{k,\ell}v_{k,\ell}v_{k,\ell}^*=R$, so the $dL=q$ vectors
$v_{k,\ell}$, relabeled as $v_1,\ldots,v_q$, have the required properties.
\end{proof}

The second preliminary fact converts an approximate-rank lower bound into a
Hilbert--Schmidt lower bound.

\begin{lemma}\label{lem:approx-rank-sv}
For any $A\in \Mn$ and $\eps>0$,
\[
  \norm{A}_2^2\ge\eps^2 \rank_{\eps}(A)\norm{A}^2.
\]
\end{lemma}

\begin{proof}
Let $k=\rank_{\eps}(A)$. If $k=0$, there is nothing to prove, so assume $k\ge 1$. Let
$A=\sum_{j=1}^ns_j(A)\,u_jv_j^*$
be a singular-value decomposition of $A$ with $s_1(A)\ge\cdots\ge s_n(A)$.  The truncated sum
$A_{k-1}=\sum_{j=1}^{k-1}s_j(A)u_jv_j^*$ has rank at most $k-1$ and
$\|A-A_{k-1}\|=s_{k}(A)$. Since 
$\rank_{\eps}(A)=k$, no matrix of rank at most $k-1$ lies within distance $\eps
\norm{A}$ of $A$, and hence $s_{k}(A)>\eps
\norm{A}$.
Therefore
\[
\norm{A}_2^2=\sum_{j=1}^n s_j(A)^2 \ge k \, s_k(A)^2 >k \eps^2\norm{A}^2. \qedhere
\]
\end{proof}

We now combine hollowization, vector partitioning, and finite-block gluing to
prove the exclusion theorem.

\begin{proof}[Proof of Theorem~\ref{thm:rank-exclusion}]
Hypothesis (b) forces $A\neq0$, and both that hypothesis and the conclusion are unchanged when $A$ is multiplied by a nonzero scalar, so we may assume $\norm{A}=1$.  Let $\Cgl$ be the constant in Lemma~\ref{lem:finite-gluing}.  We will choose a fixed integer $r_0$ and then the constants $\delta_0,\varepsilon_0,K_0$, in that order.

Suppose, toward a contradiction, that $\rank_{\varepsilon_0}(A)\le\delta_0n$.  By the definition of approximate rank, there are matrices $F,R\in\Mn$ such that
\[
 A=F+R,
 \qquad
 \rank F\le\delta_0n,
 \qquad
 \|R\|\le\varepsilon_0.
\]
Let $P$ be the orthogonal projection onto $\ran F+\ran F^*$.  Since $\ran F\subseteq\ran P$, we have $PF=F$.  Since $\ran F^*\subseteq\ran P$, taking adjoints gives $FP=F$.  Hence $F=PFP$.  Moreover,
\[
 \rank P\le\rank F+\rank F^*\le2\delta_0n,
 \qquad
 \|F\|\le\|A\|+\|R\|\le1+\varepsilon_0.
\]
If $P=0$, then $F=0$ and $1=\|A\|\le\varepsilon_0$, which will be excluded by our choice of $\varepsilon_0$.  Thus $P\neq0$.  Since $\tau(P)=\tfrac{\rank P}{n}\le2\delta_0$, Theorem~\ref{thm:hollow}(ii) applies to the three traceless Hermitian matrices
\[
 \operatorname{Re}A,
 \qquad
 \operatorname{Im}A,
 \qquad
 P-\tau(P)I.
\]
Conjugate $A,F,R,P$ simultaneously by the resulting unitary matrix.  All norm, rank, and support relations are preserved.  The first two transformed matrices are hollow, hence $A$ is hollow.  After reordering coordinates, the third conclusion says
\[
 \langle Pe_i,e_i\rangle=\tau(P)
 \quad\text{for }i\in [n-2].
\]
For $i\in [n-2]$, let $u_i=Pe_i\in\ran P$.  Then $\|u_i\|^2=\tau(P)\le2\delta_0$, and on $\ran P$,
\[
 \sum_{i=1}^{n-2}u_iu_i^*
 =P P_{\{1,\ldots,n-2\}}P
 \preceq P|_{\ran P}=I_{\ran P}.
\]
Lemma~\ref{lem:frame-tools}, applied in $\ran P$ with $\delta=2\delta_0$,
supplies $q=\rank P\,\lceil(2\delta_0)^{-1}\rceil$ vectors
$v_1,\ldots,v_q\in\ran P$ of squared norm at most $2\delta_0$ such that
$u_1,\ldots,u_{n-2},v_1,\ldots,v_q$ is a Parseval frame of $\ran P$.

Index this frame as $w_1,\ldots,w_{n-2+q}$, with $w_i=u_i$ for
$i=1,\ldots,n-2$ and $w_{n-2+j}=v_j$ for $j=1,\ldots,q$.
Theorem~\ref{thm:mss} with $r_0$ classes gives a partition
$\{T_1,\ldots,T_{r_0}\}$ of $[n-2+q]$ such that
\[
 \Bigl\|\sum_{i\in T_j}w_iw_i^*\Bigr\|
 \le\left(\frac1{\sqrt{r_0}}+\sqrt{2\delta_0}\right)^2=:\beta
 \quad\text{for }j=1,\ldots,r_0.
\]
Put $S_j=T_j\cap[n-2]$.  Then $\{S_1,\ldots,S_{r_0}\}$ is a partition of $[n-2]$, some of whose classes may be empty.  Since $0\preceq\sum_{i\in S_j}u_iu_i^*\preceq\sum_{i\in T_j}w_iw_i^*$,
\[
 \|PP_{S_j}P\|
 =\Bigl\|\sum_{i\in S_j}u_iu_i^*\Bigr\|
 \le\beta
 \quad\text{for }j=1,\ldots,r_0.
\]
Since
$\|P_{S_j}P\|^2
 =\|(P_{S_j}P)^*(P_{S_j}P)\|
 =\|PP_{S_j}P\|$,
we have $\|P_{S_j}P\|^2\le\beta$.  Together with the identity $F=PFP$, we obtain
\[
 \begin{aligned}
 \|P_{S_j}FP_{S_j}\|
 &=\|(P_{S_j}P)F(PP_{S_j})\|\\
 &\le\|P_{S_j}P\|^2\|F\|\\
 &\le(1+\varepsilon_0)\beta.
 \end{aligned}
\]
The remainder contributes at most $\varepsilon_0$, and therefore
\[
 \|P_{S_j}AP_{S_j}\|
 \le(1+\varepsilon_0)\beta+\varepsilon_0
 =:\eta.
\]

Use the sets $S_1,\ldots,S_{r_0}$ and the two exceptional singleton coordinates as an orthogonal block partition, discarding empty sets.  Since $A$ is hollow, every coordinate compression is traceless.  In particular, both singleton compressions vanish.  Every nonempty $S_j$ has size at most $n-2$, so by hypothesis (a),
\[
 \muc(P_{S_j}AP_{S_j})
 \le K\|P_{S_j}AP_{S_j}\|
 \le K\eta.
\]
Lemma~\ref{lem:finite-gluing}, applied to at most $r_0+2$ blocks, gives $\muc(A)\le\Cgl\sqrt{r_0+2}\,(1+K\eta)$.

We now make the choices explicit.  First choose $r_0$ so large that
$\frac{2\Cgl\sqrt{r_0+2}}{r_0}\le\frac14$.
Next choose $\delta_0>0$ sufficiently small that $2\delta_0<1$ and
$\left(\frac1{\sqrt{r_0}}+\sqrt{2\delta_0}\right)^2
 \le\frac{3}{2r_0}$.
Finally choose $0<\varepsilon_0<1/2$ so small that
$(1+\varepsilon_0)\frac{3}{2r_0}+\varepsilon_0
 \le\frac2{r_0}$.
For these choices, $\eta\le2/r_0$, and hence
$\Cgl\sqrt{r_0+2}\,\eta\le\frac14$.
Choose
$K_0=4\Cgl\sqrt{r_0+2}$.
If $K\ge K_0$, then the constant term and the coefficient of $K$ in the gluing estimate are both at most $K/4$.  Thus
\[
 \muc(A)\le K/4+K/4=K/2,
\]
contradicting $\muc(A)>3K/4$.  This proves
$\rank_{\varepsilon_0}(A)>\delta_0n$. Lemma~\ref{lem:approx-rank-sv} thus implies 
\[
\|A\|_2^2\ge\varepsilon_0^2\rank_{\varepsilon_0}(A)>\delta_0\varepsilon_0^2n. \qedhere
\]
\end{proof}

\subsection{Stable rank and trace polarization}

The approximate-rank argument has supplied a Hilbert--Schmidt lower bound.
To apply Section~\ref{sec:links}, we still need a trace gap between two
subspaces of equal dimension.  The next lemma makes this conversion for
any traceless matrix of even dimension satisfying the lower bound.
At least one of $\operatorname{Re}A$ and $\operatorname{Im}A$ has large
trace norm.  Since that Hermitian matrix is traceless, its positive and
negative spectral masses are equal.  Selecting its largest half of the
eigenvalues captures a fixed proportion of the positive mass and gives the
trace gap needed in Theorem~\ref{thm:trace-gap}.

\begin{lemma}\label{lem:trace-polarization}
Let $\sigma>0$, let $n$ be even, and let $A\in \Mn$ be traceless with
\[
 \norm{A}_2^2\ge\sigma n\norm{A}^2.
\]
Then there is an orthogonal projection $P\in\Mn$ of rank $n/2$ such that
\[
 \abs{\Tr(PA)}\ge\frac\sigma8n\norm{A}.
\]
Consequently, relative to $P\C^n\oplus(I-P)\C^n$, the normalized trace $\alpha$ of the first diagonal block satisfies
\[
 \abs{\alpha}\ge\frac\sigma4\norm{A}.
\]
\end{lemma}

\begin{proof}
The case $A=0$ is trivial, so by homogeneity we may assume $\|A\|=1$.  Every singular value then lies in $[0,1]$, and therefore
\[
 \|A\|_1=\sum_js_j(A)
 \ge\sum_js_j(A)^2
 =\|A\|_2^2
 \ge\sigma n.
\]
Write $A=H+iK$ with $H=H^*$ and $K=K^*$.  The trace-norm triangle inequality gives
$\|A\|_1\le\|H\|_1+\|K\|_1$, so at least one of the two self-adjoint matrices, denoted by $M$, satisfies $\|M\|_1\ge\tfrac\sigma2n$. Because $A$ is traceless, so are $H$ and $K$, and hence so is $M$.

Let $\lambda_1\ge\cdots\ge\lambda_n$ be the eigenvalues of $M$ and put $m=n/2$.  The total positive and negative masses agree:
\[
 S:=\sum_{\lambda_i>0}\lambda_i
 =\sum_{\lambda_i<0}|\lambda_i|
 =\frac12\|M\|_1.
\]
We claim that the largest $m$ eigenvalues have sum at least $S/2$.  If there are $p\ge m$ positive eigenvalues, their decreasing order implies that the largest $m$ carry at least the fraction $m/p\ge1/2$ of the positive mass.  If $p<m$, consider the remaining $2m-p$ nonpositive eigenvalues, including any zeros.  Their absolute values have total sum $S$.  The largest $m$ eigenvalues consist of all $p$ positive eigenvalues together with the $m-p$ nonpositive eigenvalues of smallest absolute value.  The latter contribute in absolute value at most the fraction $\tfrac{m-p}{2m-p}\le\tfrac12$ of $S$.  In either case,
\[
 \sum_{i=1}^m\lambda_i\ge\frac S2=\frac14\|M\|_1.
\]

Let $P$ be the orthogonal projection onto the span of the corresponding $m$ orthonormal eigenvectors.  Then $\Tr(PM)\ge\tfrac14\|M\|_1\ge\tfrac\sigma8n$. If $M=H$, this number is $\operatorname{Re}\Tr(PA)$; if $M=K$, it is $\operatorname{Im}\Tr(PA)$.  Hence
\[
 |\Tr(PA)|\ge\frac\sigma8n.
\]
Since $P^2=P$, cyclicity gives $\Tr(PAP)=\Tr(PA)$.  The first diagonal block has dimension $m=n/2$, so its normalized trace is $\alpha=\tfrac1m\Tr(PAP)$, $|\alpha|\ge\tfrac\sigma4$. The complementary block has the same dimension and normalized trace $-\alpha$, because the total trace is zero.
\end{proof}

Applying the balanced trace-gap theorem to this projection proves the stated
stable-rank consequence.

\begin{proof}[Proof of Corollary~\ref{cor:stable-easy}]
Lemma~\ref{lem:trace-polarization} supplies a projection $P$ of rank $n/2$ such that, relative to
$P\mathbb C^n\oplus(I-P)\mathbb C^n$, the first diagonal block has normalized trace $\alpha$ with $|\alpha|\ge(\sigma/4)\|A\|$.  Since the two blocks have equal dimension and $\Tr A=0$, the second normalized trace is $-\alpha$.  The hypotheses of Theorem~\ref{thm:trace-gap} therefore hold with $\delta=\sigma/4$, and that theorem gives the asserted bound.
\end{proof}

\section{Proof of the main theorem}\label{sec:main}

The estimates now give two incompatible requirements for a minimal
counterexample.  Theorem~\ref{thm:rank-exclusion} forces its stable rank to
be proportional to the dimension, while Corollary~\ref{cor:stable-easy}
bounds the commutator cost of every even-dimensional matrix with that
property.  Lemma~\ref{lem:one-sided}, the square-root gluing estimate from
Section~\ref{sec:gluing}, handles odd dimensions by leaving an
even-dimensional compression whose cost is still large enough for the
exclusion theorem.  We first complete this argument over $\mathbb C$, then
transfer the bound to real matrices of the same size and prove the
ultraproduct corollary.

\begin{proof}[Proof of Theorem~\ref{thm:main} over $\mathbb C$]
Let $\varepsilon_0,\delta_0,K_0$ and $c_0$ be the constants in Theorem~\ref{thm:rank-exclusion}, and let
$C_*=C_{c_0/4}$ be the constant supplied by Corollary~\ref{cor:stable-easy} for $\sigma=c_0$.  Fix one absolute number $K$ satisfying $K>\max\{K_0, 64,(\sqrt{C_*}+1)^2\}$. We prove that $\muc(A)\le K\|A\|$ for every traceless matrix $A$.

Assume the contrary.  After scaling a nonzero counterexample, there is a traceless contraction with norm one and commutator cost larger than $K$.  Choose such a matrix
$A\in \Mn$ with $n$ minimal.  Necessarily $n\ge2$, since the only
traceless scalar matrix is zero.  Then $\|A\|=1$, $\muc(A)>K$, and minimality means that every traceless matrix $D$ of dimension strictly less than $n$ satisfies $\muc(D)\le K\|D\|$. Indeed, otherwise a smaller-dimensional counterexample could also be normalized to norm one.

\emph{Even dimension.}
Suppose first that $n$ is even.  All hypotheses of Theorem~\ref{thm:rank-exclusion} are satisfied: $K\ge K_0$, the smaller-dimensional bound holds, and
$\muc(A)>K\|A\|>\tfrac34K\|A\|$.  Hence $\|A\|_2^2\ge c_0n$. Corollary~\ref{cor:stable-easy} now gives
$\muc(A)\le C_*<K$, contradicting the choice of $A$.

\emph{Odd dimension.}
Suppose now that $n$ is odd.  By Theorem~\ref{thm:hollow}(i) and unitary invariance, we may assume $A$ is hollow.  Splitting off the first coordinate gives
\[
 A=\begin{pmatrix}0&X\\Y&D\end{pmatrix},
\]
where $D\in M_{n-1}(\mathbb C)$ is traceless of even dimension.  Since $X$ and $Y$ are corners of the contraction $A$, $\max\{\|X\|,\|Y\|\}\le1$. Lemma~\ref{lem:one-sided} implies
\[
 \sqrt{\muc(A)}\le\sqrt{\muc(D)}+\sqrt{\max\{\|X\|,\|Y\|\}}
 \le\sqrt{\muc(D)}+1.
\]
Because $\muc(A)>K$, we obtain $\muc(D)>(\sqrt K-1)^2\ge \tfrac34K\ge \tfrac34K\|D\|$. Every traceless matrix of dimension strictly smaller than $n-1$ also has dimension smaller than $n$, so it satisfies the bound $\muc(E)\le K\|E\|$ by minimality of $n$.  Therefore Theorem~\ref{thm:rank-exclusion} applies to $D$ and yields $\|D\|_2^2\ge c_0(n-1)\|D\|^2$. Corollary~\ref{cor:stable-easy} gives $\muc(D)\le C_*\|D\|\le C_*$. Applying Lemma~\ref{lem:one-sided} once more,
\[
 \muc(A)
 \le(\sqrt{\muc(D)}+\sqrt{\max\{\|X\|,\|Y\|\}})^2
 \le(\sqrt{C_*}+1)^2
 <K,
\]
again a contradiction.

Both parity cases are impossible.  Hence every traceless complex matrix
satisfies $\muc(A)\le K\|A\|$.  This infimum is attained: balance the factors
in a minimizing sequence, extract a convergent subsequence in the
finite-dimensional matrix space, and pass to the limit in the commutator
identity.  Denote the resulting absolute constant by $K_{\mathbb C}$.
This proves the complex case; the reduction below supplies a universal
constant for both fields.
\end{proof}

\subsection{Real matrices}\label{sec:real-case}

We now deduce the real case from the complex theorem by choosing a complex
structure on the given real space.  In even dimensions at least
four, we choose an orthogonal real matrix $J$ with $J^2=-I$ to represent
multiplication by $i$.  The matrix $A$ splits into a part commuting with
$J$, which is complex-linear, and a part anticommuting with $J$, which is
antilinear.  Real trace zero alone does not guarantee that the first part
has zero complex trace; we choose $J$ so that $\Tr(JA)=0$ as well.

Apply the complex commutator estimate to the complex-linear part.
Shifting its first factor by a multiple of $J$ leaves that commutator
unchanged and makes the commutator equation on the antilinear part
invertible.  Solving this equation absorbs the remaining part of $A$.

The decomposition into complex-linear and antilinear parts is standard
linear algebra.  The spectral-shift method is the one used for Sylvester
equations; see Rosenblum \cite{Rosenblum}.  That reference supplies the
solvability principle, rather than the quantitative reduction below,
which we prove in full, including the choice of complex structure and the
small and odd dimensions.

\begin{proposition}\label{prop:real-case}
Suppose $K_{\mathbb C}\ge1/2$ is such that every traceless complex matrix $T$
admits complex factors of the same size with
$T=[B,C]$ and $\|B\|\,\|C\|\le K_{\mathbb C}\|T\|$.
Then every traceless $A\in M_n(\mathbb R)$ admits
$B,C\in M_n(\mathbb R)$ with $A=[B,C]$ and
\[
 \|B\|\,\|C\|\le K_{\mathbb R}\|A\|,
 \qquad
 K_{\mathbb R}=\left(\sqrt{3K_{\mathbb C}+\tfrac32}+1\right)^2.
\]
\end{proposition}

\begin{proof}
The case $A=0$ is immediate, so assume $\|A\|=1$ by homogeneity, and put
$\kappa=3K_{\mathbb C}+3/2$.  We first establish the sharper bound $\kappa$
in even dimensions $n=2m\ge4$.

\emph{Choose a complex structure.}
Write $A=H+S$, where $H^T=H$ and $S^T=-S$.  The real canonical form of $S$
gives an orthonormal basis $e_1,f_1,\ldots,e_m,f_m$ in which each plane
$\operatorname{span}\{e_j,f_j\}$ is invariant under $S$.  With cyclic
indices, define
\[
 Je_j=f_{j+1},\qquad Jf_{j+1}=-e_j
 \quad (j=1,\ldots,m).
\]
Thus $J^T=-J$ and $J^2=-I$, so $J$ is an orthogonal complex structure.
Since $m\ge2$, the nonzero entries of $J$ and $S$ occur on disjoint pairs of
coordinates, giving $\Tr(JS)=0$.  Also $\Tr(JH)=0$, because $J$ is
skew-symmetric and $H$ is symmetric.  Consequently, $\Tr(JA)=0$.

Set
\[
 L=\frac{A-JAJ}{2},\qquad E=\frac{A+JAJ}{2}.
\]
Then $JL=LJ$, $JE=-EJ$, and $\|L\|,\|E\|\le1$.  Viewing $J$ as
multiplication by $i$ identifies $\mathbb R^{2m}$ with a complex Hilbert
space of dimension $m$, with the same norm.  The operator $L$ is
complex-linear.  Cyclicity of the real trace gives
\[
 \Tr_{\mathbb R}L=\Tr_{\mathbb R}A=0,
 \qquad
 \Tr_{\mathbb R}(JL)=\Tr_{\mathbb R}(JA)=0.
\]
These are twice the real part and minus twice the imaginary part,
respectively, of $\Tr_{\mathbb C}L$.  Hence $L$ has zero complex trace.

\emph{Absorb the antilinear part.}
The assumed complex estimate supplies $L=[B_0,C_0]$, where $B_0,C_0$ are
real operators commuting with $J$.  Reciprocal rescaling gives
$\|B_0\|\le1$ and $\|C_0\|\le K_{\mathbb C}$; if $L=0$, take both
factors to be zero.  Let
\[
 \mathcal E_J=\{X\in M_n(\mathbb R):XJ=-JX\}.
\]
Both maps $X\mapsto[B_0,X]$ and $X\mapsto[2J,X]$ preserve this real
subspace.  On $\mathcal E_J$ the latter map is $X\mapsto4JX$, with inverse
$Y\mapsto-JY/4$ of norm $1/4$, whereas the former has norm at most $2$.
A Neumann series therefore shows that $X\mapsto[B_0+2J,X]$ is invertible on
$\mathcal E_J$, with inverse norm at most $1/2$.  Since $E\in\mathcal E_J$,
there is a real $X$ such that
\[
 [B_0+2J,X]=E,\qquad \|X\|\le\tfrac12.
\]
As $C_0$ commutes with $J$, this gives
\[
 A=[B_0+2J,C_0+X],
 \qquad
 \|B_0+2J\|\,\|C_0+X\|
 \le3\left(K_{\mathbb C}+\tfrac12\right)=\kappa.
\]

\emph{Small and odd dimensions.}
For any traceless real matrix, the symmetric part is traceless, so its
quadratic form vanishes at some real unit vector.  Taking that vector first
in an orthonormal basis makes the first diagonal entry of the matrix zero.
For $n=2$, the other diagonal entry is then zero as well, and
\[
 \begin{pmatrix}0&x\\y&0\end{pmatrix}
 =\left[
 \begin{pmatrix}1/2&0\\0&-1/2\end{pmatrix},
 \begin{pmatrix}0&x\\-y&0\end{pmatrix}
 \right]
\]
has factor norm product $\|A\|/2\le\kappa$.  The case $n=1$ is trivial.

If $n\ge3$ is odd, the same orthogonal change of basis gives
$A=\left(\begin{smallmatrix}0&U\\V&D\end{smallmatrix}\right)$ with $D$
traceless of even dimension and $\|D\|,\|U\|,\|V\|\le1$.
The even-dimensional result and reciprocal rescaling provide real factors
$D=[B_0,C_0]$ with $\|B_0\|,\|C_0\|\le\sqrt\kappa$.
We use the real resolvent construction underlying
Lemma~\ref{lem:one-sided}.  Put $t=\sqrt\kappa+1$ and define
\[
 \widehat B=\begin{pmatrix}t&0\\0&B_0\end{pmatrix},\qquad
 \widehat C=\begin{pmatrix}
 0&U(tI-B_0)^{-1}\\
 (B_0-tI)^{-1}V&C_0
 \end{pmatrix}.
\]
Both inverses are real and have norm at most $1$.  Direct multiplication
gives $[\widehat B,\widehat C]=A$.  The off-diagonal part of $\widehat C$
has norm at most $1$, and hence
\[
 \|\widehat B\|\,\|\widehat C\|
 \le(\sqrt\kappa+1)^2=K_{\mathbb R}.
\]
Undoing the normalization proves the proposition.
\end{proof}

Applying Proposition~\ref{prop:real-case} to the complex bound proved above,
and choosing $K=K_{\mathbb R}$, completes the proof of
Theorem~\ref{thm:main} for both fields.

\subsection{Commutators in tracial ultraproducts}\label{sec:ultraproducts}

With the matrix theorem proved over both fields, we finish with the
standard ultraproduct consequence identified by Johnson, Ozawa, and
Schechtman \cite[Concluding remarks, item~4]{JOS}.  The argument uses only
the uniform matrix bound: solve the equation coordinatewise and balance
the factors so that they define bounded sequences.  Dykema and Skripka
\cite[proof of Theorem~2.2]{DykemaSkripka} use this passage to the Wright
factor for normal elements.

To preserve the constant $K$, we choose a representative with the correct
norm before subtracting its scalar trace.  Retain the notation from
Section~\ref{sec:intro-main}, and write $\pi$ for the quotient map onto
$\mathcal M$.  We use the standard estimate
$\|\pi((a_k))\|\le\lim_\omega\|a_k\|$.

\begin{proof}[Proof of Corollary~\ref{cor:ultraproduct}]
The case $T=0$ is immediate, so assume $\|T\|=1$ by homogeneity.
Choose a bounded representative $(a_k)$ and set
$S_k=a_kf(a_k^*a_k)$, where $f(t)=1$ for $0\le t\le1$ and
$f(t)=t^{-1/2}$ for $t\ge1$.  This clips each singular value at $1$, so
$\|S_k\|\le1$, while continuous functional calculus gives
$\pi((S_k))=Tf(T^*T)=T$.  Over $\mathbb R$, the real-valued function $f$
preserves real matrices, so the same construction applies.

Now put $t_k=\tau_k(S_k)\in\mathbb F$.  Then
$|t_k|\le1$ and $\lim_\omega t_k=\tau_\omega(T)=0$.  Thus the scalar
sequence $(t_kI_{n_k})$ vanishes in the quotient, and
$T_k=S_k-t_kI_{n_k}$ is a traceless representative of $T$ satisfying
$\|T_k\|\le1+|t_k|$.

Apply Theorem~\ref{thm:main} over $\mathbb F$ to each $T_k$, and balance the
two factors by reciprocal rescaling.  This gives
$B_k,C_k\in M_{n_k}(\mathbb F)$ with
\[
 T_k=[B_k,C_k],
 \qquad
 \|B_k\|=\|C_k\|\le\sqrt{K\|T_k\|}
 \le\sqrt{K(1+|t_k|)}.
\]
When $T_k=0$, take both factors to be zero.  These sequences are bounded
by $\sqrt{2K}$ and therefore define $B=\pi((B_k))$ and $C=\pi((C_k))$.
Since $\pi$ is a homomorphism,
\[
 [B,C]=\pi(([B_k,C_k]))=\pi((T_k))=T.
\]
Finally, $|t_k|\to0$ along $\omega$, so
\[
 \|B\|\le\lim_\omega\|B_k\|\le\sqrt K,
 \qquad
 \|C\|\le\lim_\omega\|C_k\|\le\sqrt K,
\]
and hence $\|B\|\,\|C\|\le K=K\|T\|$.
\end{proof}

\section*{Acknowledgments}

We acknowledge the closely contemporaneous work of Shen, Wang, and Zhi
\cite{ShenWangZhi}, whose first preprint, dated September~9, 2026, proves
the same dimension-free bound for complex matrices.  Our independent
approach was already recorded in drafts dated August~14, 2026.  Those
drafts contain the core complex argument: the trace-class Poincar\'e
inequality, local regularity of the commutator map, the invertible-link
estimate, and the approximate-rank argument.  We record these dates to
document independent development; we make no claim of priority.  The
proof here does not use their results.

\medskip
\noindent\textbf{Use of AI tools.}
During the preparation of this work, the author used OpenAI's ChatGPT-5.6 Sol and 6.0 Astra to explore ideas, assist with exposition, conduct literature searches, and check mathematical arguments. The author assumes full responsibility for the content, validity, and attribution of all mathematical claims.


\begin{thebibliography}{30}
\setlength{\itemsep}{0.3em}

\bibitem{AlbertMuckenhoupt}
A.~A. Albert and B.~Muckenhoupt,
\emph{On matrices of trace zero},
Michigan Math. J. \textbf{4} (1957), 1--3.

\bibitem{Anderson}
J.~Anderson,
\emph{Extensions, restrictions, and representations of states on $C^*$-algebras},
Trans. Amer. Math. Soc. \textbf{249} (1979), no.~2, 303--329.

\bibitem{AngelSchechtman}
O.~Angel and G.~Schechtman,
\emph{The Hilbert--Schmidt version of the commutator theorem for zero trace matrices},
Bull. Lond. Math. Soc. \textbf{47} (2015), no.~4, 715--719.


\bibitem{AuYeungPoon}
Y.~H. Au-Yeung and Y.~T. Poon,
\emph{A remark on the convexity and positive definiteness concerning Hermitian matrices},
Southeast Asian Bull. Math. \textbf{3} (1979), 85--92.

\bibitem{BST}
A.~Ben-Aroya, O.~Schwartz, and A.~Ta-Shma,
\emph{Quantum expanders: motivation and construction},
Theory Comput. \textbf{6} (2010), 47--79.

\bibitem{BDM}
R.~Bhatia, C.~Davis, and A.~McIntosh,
\emph{Perturbation of spectral subspaces and solution of linear operator equations},
Linear Algebra Appl. \textbf{52/53} (1983), 45--67.

\bibitem{BhatiaRosenthal}
R.~Bhatia and P.~Rosenthal,
\emph{How and why to solve the operator equation $AX-XB=Y$},
Bull. Lond. Math. Soc. \textbf{29} (1997), no.~1, 1--21.


\bibitem{DammFassbender}
T.~Damm and H.~Fa\ss bender,
\emph{Simultaneous hollowization, joint numerical range, and stabilization by noise},
SIAM J. Matrix Anal. Appl. \textbf{41} (2020), no.~2, 637--656.


\bibitem{DykemaSkripka}
K.~Dykema and A.~Skripka,
\emph{On single commutators in $\mathrm{II}_1$-factors},
Proc. Amer. Math. Soc. \textbf{140} (2012), no.~3, 931--940.

\bibitem{Fillmore}
P.~A. Fillmore,
\emph{On similarity and the diagonal of a matrix},
Amer. Math. Monthly \textbf{76} (1969), 167--169.

\bibitem{Hastings}
M.~B. Hastings,
\emph{Random unitaries give quantum expanders},
Phys. Rev. A \textbf{76} (2007), 032315.

\bibitem{JOS}
W.~B. Johnson, N.~Ozawa, and G.~Schechtman,
\emph{A quantitative version of the commutator theorem for zero trace matrices},
Proc. Natl. Acad. Sci. USA \textbf{110} (2013), no.~48, 19251--19255.

\bibitem{MSS}
A.~W. Marcus, D.~A. Spielman, and N.~Srivastava,
\emph{Interlacing families II: mixed characteristic polynomials and the Kadison--Singer problem},
Ann. of Math. (2) \textbf{182} (2015), no.~1, 327--350.

\bibitem{PearcyTopping}
C.~Pearcy and D.~Topping,
\emph{Commutators and certain $\mathrm{II}_1$-factors},
J. Funct. Anal. \textbf{3} (1969), 69--78.

\bibitem{Pisier}
G.~Pisier,
\emph{Quantum expanders and geometry of operator spaces},
J. Eur. Math. Soc. (JEMS) \textbf{16} (2014), no.~6, 1183--1219.

\bibitem{Rosenblum}
M.~Rosenblum,
\emph{On the operator equation $BX-XA=Q$},
Duke Math. J. \textbf{23} (1956), 263--269.

\bibitem{RS}
M.~Ravichandran and N.~Srivastava,
\emph{Asymptotically optimal multi-paving},
Int. Math. Res. Not. IMRN (2021), no.~14, 10908--10940.

\bibitem{Ricard}
\'E.~Ricard,
\emph{H\"older estimates for the noncommutative Mazur maps},
Arch. Math. (Basel) \textbf{104} (2015), no.~1, 37--45.

\bibitem{ShenWangZhi}
H.~Shen, J.~Wang, and L.~Zhi,
\emph{A dimension-independent commutator bound}. arXiv:2609.09938.

\bibitem{Shoda}
K.~Shoda,
\emph{Einige S\"atze \"uber Matrizen},
Jpn. J. Math. \textbf{13} (1936), 361--365.

\bibitem{WenFangYao}
S.~Wen, J.~Fang, and Z.~Yao,
\emph{A stronger version of Dixmier's averaging theorem and some applications},
J. Funct. Anal. \textbf{287} (2024), no.~8, 110569.


\bibitem{Wright}
F.~B. Wright,
\emph{A reduction for algebras of finite type},
Ann. of Math. (2) \textbf{60} (1954), 560--570.

\end{thebibliography}
\end{document}